\documentclass[a4paper,12pt]{amsart}
\usepackage[a4paper,top=1.5cm,bottom=1.5cm,left=1.5cm,right=1.5cm,marginparwidth=1.75cm]{geometry}
\usepackage{amsmath}
\usepackage{amssymb}
\usepackage{amsthm}
\usepackage[a4paper]{geometry}
\usepackage{todonotes}
\usepackage{enumerate}   
\usepackage{cite}
\usepackage{thmtools}
\usepackage{thm-restate}

\theoremstyle{plain}
\newtheorem{thm}{Theorem}[section]
\newtheorem{cor}{Corollary}[thm]
\newtheorem{lem}[thm]{Lemma}
\newtheorem{prop}[thm]{Proposition}

\theoremstyle{definition}

\theoremstyle{remark}
\newtheorem*{rem}{Remark}

\newcommand\inner[2]{\langle #1, #2 \rangle}
\newcommand{\parcheck}{\check{\phantom{x}}}	
\newcommand{\tco}{\mathcal{T}}
\newcommand{\bo}{B(L^2(\mathbb{R}^{d}))}

\newcommand{\N}{\mathbb{N}}
\newcommand{\R}{\mathbb{R}}
\newcommand{\Rd}{\mathbb{R}^d}
\newcommand{\Rdd}{\mathbb{R}^{2d}}

\newcommand{\states}{\mathcal{W}}

\newcommand{\tr}{\mathrm{tr}}

\begin{document}

\pagestyle{plain}
\title{On accumulated Cohen's class distributions and mixed-state localization operators}
\author{Franz Luef}
\author{Eirik Skrettingland} 
\address{Department of Mathematics\\ NTNU Norwegian University of Science and
Technology\\ NO–7491 Trondheim\\Norway}
\email{franz.luef@math.ntnu.no, eirik.skrettingland@ntnu.no}
\begin{abstract}
Recently we introduced mixed-state localization operators associated to a density operator and a (compact) domain in phase space. We continue the investigations of their eigenvalues and eigenvectors. Our main focus is the definition of a time-frequency distribution which is based on the Cohen class distribution associated to the density operator and the eigenvectors of the mixed-state localization operator. This time-frequency distribution is called the accumulated Cohen class distribution. If the trace class operator is a rank-one operator, then the mixed-state localization operators and the accumulated Cohen class distribution reduce to Daubechies' localization operators and the accumulated spectrogram. We extend all the results about the accumulated spectrogram to the accumulated Cohen class distribution. The techniques used in the case of spectrograms cannot be adapted to other distributions in Cohen's class since they rely on the reproducing kernel property of the short-time Fourier transform. Our approach is based on quantum harmonic analysis on phase space which also provides the tools and notions to introduce the analogues of the accumulated spectrogram for mixed-state localization operators; the accumulated Cohen's class distributions.   
\end{abstract}
\maketitle \pagestyle{myheadings} \markboth{F. Luef and E. Skrettingland}{On accumulated Cohen's class distributions}
\thispagestyle{empty}

\section{Introduction}

In their study of the spectral behavior of localization operators Abreu et al. introduced the accumulated spectrogram and established interesting results in \cite{Abreu:2016,Abreu:2017}, which revealed some intriguing features of localization operators. We show how the theorems in \cite{Abreu:2016,Abreu:2017} may be extended to a setting involving (infinite) sums of localization operators, aka mixed-state localization operators. 

The main object of this paper is an in-depth treatment of the mixed-state localization operators from \cite{Luef:2018} and their associated time-frequency distributions from the perspective developed in \cite{Abreu:2016,Abreu:2017}, and to describe them we first recall some facts about quantum harmonic analysis \cite{Werner:1984,Luef:2017vs}. Concretely, the convolution between two trace class operators and the convolution between a function and a trace class operator. Both convolutions are defined in terms of the {\it translation} of an operator $S$ by a point $z=(x,\omega)$ in phase space $\Rdd$: 
\begin{equation*}
  \alpha_z(S)=\pi(z)S\pi(z)^*,
\end{equation*}
where $\pi(z)$ denotes the time-frequency shift of $\psi\in L^2(\R^d)$ by $z=(x,\omega)\in\Rdd$, $\pi(z)\psi(t)=e^{2\pi it\omega}\psi(t-x)$.
The convolution between two trace class operators $S$ and $T$ is the function on $\Rdd$ given by 
\begin{equation*}
  S \star T(z) = \tr(S\alpha_z (\check{T})) \hskip 1em \text{for } z\in \Rdd,
\end{equation*}
where $\check{T}=PTP$ for $P\psi(x)=\psi(-x)$. An interesting example is the convolution of rank-one operators: 
\[(\psi\otimes \varphi)\star (\check{\psi} \otimes \check{\varphi})(z)= |V_{\varphi} \psi(z)|^2,\] 
where $\check{\psi}=P\psi$ and $\psi\otimes \varphi$ is given by $\psi\otimes \varphi(\xi)=\langle\xi,\varphi\rangle\psi$.

The convolution between a function  $f\in L^1(\R^{2d})$ 
and a trace class operator  $S$ is given by 
\begin{equation*}
  f\star S := \iint_{\R^{2d}}f(z)\alpha_z(S) \ dz;
\end{equation*}

For a rank-one operator $S=\varphi_2 \otimes \varphi_1$ with $\varphi_1,\varphi_2\in L^2(\Rd)$, we have that
	\[f\star (\varphi_2 \otimes \varphi_1)(\psi)=\iint_{\mathbb{R}^{2d}}f(z)V_{\varphi_1}\psi(z)\pi(z)\varphi_2\,dz, \]
which is a STFT-multiplier \cite{Hlawatsch:1992}, also known as a localization operator. In the case of $f=\chi_\Omega$, the characteristic function of a measurable subset $\Omega$ of $\R^{2d}$, and $\varphi_1=\varphi_2$ we obtain Daubechies' localization operator $\mathcal{A}_\Omega^{\varphi}$ \cite{Daubechies:1988}. Interesting results on the relation between the eigenfunctions of a localization operator and its domain have been given in \cite{Abreu:2012}.

A \textit{mixed-state localization operator} is an operator of the form $\chi_{\Omega}\star S$, where $S$ is a positive trace class operator with $\tr(S)=1$ -- a \textit{density operator}. The main theme of our paper is the step from rank-one operators $\varphi \otimes\varphi$ to arbitrary density operators, i.e. the step from Daubechies' localization operators to mixed-state localization operators. 

The quadratic time-frequency representation associated to localization operators is the spectrogram $|V_\varphi \psi(z)|^2$. In order to extend the results in \cite{Abreu:2016,Abreu:2017} to mixed-state localization operators we have to find a quadratic time-frequency representation defined by the density operator $S$. It turns out that elements of Cohen's class provide the desired object. 

 We have shown in \cite{Luef:2018} that $Q$ belongs to Cohen's class if it is of the form $Q_S(\psi)=\check{S}\star(\psi \otimes\psi)$, where $S$ is a linear operator mapping the Schwartz class $\mathcal{S}(\R^{2d})$ to the space of tempered distributions $\mathcal{S}^\prime(\R^{2d})$. In particular, density operators $S$ provide distributions in Cohen's class. The relevance of Cohen's class distributions has also been noted by \cite{Boggiatto:2017,Boggiatto:2018,Boggiatto:2010,Ramanathan:1994}.  

Furthermore we have given the following characterization in \cite{Luef:2018}: $S$ is a density operator if and only if $Q_S(\psi)$ is a positive function and $\int_{\R^{2d}}Q_S(\psi)(z)dz=\|\psi\|^2_2$ for any $\psi\in L^2(\R^{2d})$. Note that $Q_{\varphi \otimes\varphi}(\psi)$ is the spectrogram $|V_\varphi \psi(z)|^2$ and thus $Q_S$ is the correct generalization of the spectrogram.

Since the mixed-state localization operator $\chi_\Omega\star S$ is a positive trace class operator, the spectral theorem yields the existence of 
a sequence of eigenvalues and of eigenfunctions. We will denote the eigenvalues of $\chi_\Omega \star S$ by $\{\lambda_k^\Omega\}_{k\in \N}$ and the orthonormal basis formed by its eigenfunctions by $\{h_k^\Omega\}_{k\in \N}$, thus the spectral representation is
\begin{equation} \label{eq:singvalgenloc}
	\chi_\Omega \star S = \sum_{k=1}^\infty \lambda_k^\Omega h_k^\Omega \otimes h_k^\Omega.
\end{equation}
We always assume that the eigenvalues are arranged in decreasing order, i.e. $\lambda_1^\Omega\geq \lambda_2^\Omega \geq \dots$. 

Quantum harmonic analysis seems to provide the natural setting for the investigations of eigenvalues and eigenvectors of (mixed-state) localization operators as in this setup many of the proofs in \cite{Abreu:2016,Abreu:2017,DeMari:2002} become natural statements about convolutions between operators. An important aspect of this paper is that one can reformulate the results of \cite{Abreu:2016} in terms of quantum harmonic analysis which then allows us to formulate their results for mixed-state localization operators. Note that our approach provides an alternative proof of results for the accumulated spectrogram as well.  

Let us briefly present our results: The first result is that the eigenvalues of a mixed-state localization operator has the same asymptotic behaviour as the one for localization operators \cite{DeMari:2002,Ramanathan:1994}, see theorem \ref{thm:locopeigenvalues}. This is a prerequisite for generalizing the results \cite{Abreu:2016,Abreu:2017}. A key fact is that the  approach in \cite{Abreu:2016,Abreu:2017} is only feasible in the case of rank-one operators. For a general density operator one has to develop a different strategy. 
Ours is based on noting that the reproducing kernel techniques can be bypassed if one notes that the replacement of the spectrogram in this case is the function $\tilde{S}=S\star \check{S}$ on phase space $\Rdd$, which reduces to the spectrogram for $S=\varphi \otimes\varphi$. 
A crucial observation is an intrinsic link between mixed-state localization operators and Cohen class distributions: 

	\begin{equation*}
		\chi_\Omega \ast \tilde{S}(z)=\sum_{k=1}^{\infty} \lambda_k^\Omega Q_S(h_k^\Omega)(z), \hskip 1em \text{for } z\in \Rdd.
	\end{equation*}
We are now in the position to introduce the accumulated spectrogram associated to a mixed-state localization operator $\chi_\Omega\star S$ for a compact set $\Omega \subset \Rdd$. The 
\textit{accumulated Cohen class distribution} is defined by  
\begin{equation*}
  \rho^S_{\Omega}(z):=\sum_{k=1}^{A_\Omega} Q_S(h_k^\Omega) \hskip 1em \text{ for } z\in \Rdd,
\end{equation*}
where $A_{\Omega}=\lceil |\Omega| \rceil$. 

Note that $\rho^{\psi\otimes\psi}_{\Omega}$ is the accumulated spectrogram which is an intriguing object both from a mathematical and application point of view. Our main results are the extension of the theorems in \cite{Abreu:2016,Abreu:2017} on the accumulated spectrogram to accumulated Cohen's class distributions. Our proofs are non-trivial adaptations of the ones in \cite{Abreu:2016,Abreu:2017} and we have tried to emphasize the modifications required by the mixed-state setting. 

In theorem \ref{thm:asymp-conv} we demonstrate the asymptotic convergence of accumulated Cohen class distributions to the characteristic function of the domain:

\begin{restatable}[Asymptotic convergence]{thm}{asymptotic} \label{thm:asymp-conv}
	Let $S$ be a density operator and $\Omega\subset \Rdd$ a compact domain. Then $$\|\rho^S_{R \Omega}(R \hskip 0.1em\boldsymbol{\cdot} )- \chi_{\Omega}\|_{L^1}\to 0 \text{ as } R\to \infty.$$ 
\end{restatable}

We then move on to study the non-asymptotic convergence of accumulated Cohen class distributions, where the bounds depend on the size of the perimeter of the domain $\Omega \subset \Rdd$. To quantify the size of the perimeter, we will use the variation of its characteristic function $\chi_\Omega$ and a subset $M_{op}^*$ of density operators: 
\begin{equation*} 
  M_{op}^*=\{S \text{ trace class operator} : S\geq 0, \tr(S)=1  \text{ and } \int_{\Rdd} \tilde{S}(z) |z| \ dz < \infty  \},
\end{equation*}
where $|z|$ is the Euclidean norm of $z$, with the associated norm $\|S\|^2_{M_{op}^*}=\int_{\Rdd} \tilde{S}(z) |z| \ dz.$
This norm lets us bound the approximation of $\chi_\Omega$ by $\chi_\Omega \ast \tilde{S}$. Consequently, we are able to prove the next statement: 
\begin{restatable}[Non-asymptotic convergence]{thm}{nonasymptotic}  \label{thm:weakl2}
	If $S\in M_{op}^*$ and $\Omega \subset \Rdd$ is a compact domain with finite perimeter such that $\|S\|^2_{M_{op}^*}|\partial \Omega|\geq 1$, then for any $\delta >0$ $$\left| \left\{ z\in \Rdd : \left| \rho_\Omega^S(z)-\chi_\Omega(z)  \right|> \delta \right\} \right|\lesssim \frac{1}{\delta^2}\|S\|^2_{M_{op}^*}|\partial \Omega|.$$
\end{restatable}

In \cite{Abreu:2017} the sharpness of this bound for the spectrogram was shown by considering Euclidean balls $B(z,R)=\{z'\in \Rdd : |z|<R\}$ as the domain $\Omega$. In theorem \ref{thm:sharprate} we demonstrate this sharpness for accumulated Cohen class distributions $Q_S$ for $S\in M_{op}^*$. Our approach is inspired by the spectrogram results in \cite{DeMari:2002,Feichtinger:2014} where the projection functional enters in a crucial manner. We give an expression for this projection functional applied to $\chi_\Omega \star S$:
 \[ \tr(\chi_\Omega \star S)-\tr((\chi_\Omega \star S)^2)=\int_{\Omega}\int_{\Rdd\setminus \Omega } \tilde{S}(z-z')\ dz' dz. \]
The results above also shed some light on results in \cite{Luef:2018}, where we considered the question of recovering $\Omega$ from $\chi_\Omega \star S$. The approach in \cite{Luef:2018} was only concerned with establishing conditions on $S$ for this to be possible, and offered no clue as to how $\Omega$ could be recovered. Theorem \ref{thm:weakl2} shows that $\rho_\Omega^S$, defined using a finite number of eigenfunctions of $\chi_\Omega \star S$, estimates $\chi_\Omega.$The sharpness of the bounds is contained in theorem \ref{thm:sharprate}:
\begin{restatable}[Sharpness]{thm}{sharpness} \label{thm:sharprate}
	Let $S\in M_{op}^*$. There exist constants $C_S^1$ and $C_S^2$ such that for $R>1$ $$C_S^1 R^{2d-1}\leq \| \rho_{B(0,R)}^S-\chi_{B(0,R)} \|_{L^1} \leq C_S^2 R^{2d-1}.$$
\end{restatable}
We close this paper by discussing some examples of Cohen's class distributions suitable for the accumulated Cohen's class construction, namely those given by a density operator $S$. In particular we show that any such distribution can be used to obtain \textit{new} examples by convolving an operator with a positive function, previously noted in a different setting by Gracia-Bond\'ia and V\'arilly \cite{Gracia:1988}.

\section{Preliminaries} \label{sec:prelim}

\subsection{The short-time Fourier transform}
If $\psi:\R^d\to \mathbb{C}$ and $z=(x,\omega)\in \R^{2d}$, we define the \textit{translation operator} $T_x$ by $T_x\psi (t)=\psi(t-x)$, the \textit{modulation operator} $M_{\omega}$ by $M_{\omega}\psi (t)=e^{2 \pi i \omega \cdot t} \psi (t)$ and the \textit{time-frequency shifts} $\pi(z)$ by $\pi(z)=M_{\omega}T_x$. 
For $\psi,\phi \in L^2(\R^d)$ the \textit{short-time Fourier transform} (STFT) $V_{\phi}\psi$ of $\psi$ with window $\phi$ is the function on $\R^{2d}$
 defined by
 \begin{equation*}
  V_{\phi}\psi(z)=\inner{\psi}{\pi(z)\phi} \hskip 1em \text{ for } z\in \Rdd,
\end{equation*}
where $\inner{\cdot}{\cdot}$ is the usual inner product on $L^2(\Rd)$. By replacing the inner product above with a duality bracket\footnote{Which we always assume is antilinear in the second coordinate, to be consistent with the inner product on $L^2(\Rd)$.}, the STFT may be extended to other spaces, such as $\psi\in \mathcal{S}(\Rd), \phi \in \mathcal{S}^{\prime}(\Rd)$ where $\mathcal{S}(\Rd)$ is the Schwartz space and $\mathcal{S}^{\prime}(\Rd)$ its dual space of tempered distributions. We will also meet a close relative of the STFT: the cross-Wigner distribution, defined for $\psi,\xi \in L^2(\Rd)$ by $$ W(\psi,\xi)(x,\omega)=\int_{\Rd} \psi\left(x+\frac{t}{2}\right)\overline{\xi\left(x-\frac{t}{2}\right)} e^{-2\pi i \omega \cdot t} \ dt \hskip 1em \text{ for } (x,\omega)\in \Rdd. $$ 

\subsection{Operator theory} \label{sec:operatortheory}
Our approach relies heavily on properties of the bounded operators $\bo$ on $L^2(\Rd)$, and a basic result is the \textit{spectral representation} of self-adjoint compact operators  \cite[Thm. 3.5]{Busch:2016}.
\begin{prop}
\label{prop:singval}
	Let $S$ be a self-adjoint, compact operator on $L^2(\Rd)$ with eigenvalues $\{\lambda_k\}_{k\in \N}$. There exists an orthonormal basis  $\{\varphi_k\}_{k\in \mathbb{N}}$ in $L^2(\Rd)$ such that $S$ may be expressed as 
\begin{equation*}
S = \sum\limits_{k \in \mathbb{N}} \lambda_k \varphi_k\otimes \varphi_k, 
\end{equation*}
with convergence in the operator norm. Here $\varphi_k\otimes \varphi_k$ is the rank-one operator defined by $\varphi_k\otimes \varphi_k (\xi)=\inner{\xi}{\varphi_k}\varphi_k$ for $\xi\in L^2(\Rd)$.
\end{prop}
\subsubsection{The trace and trace class operators}
For a \textit{positive} operator $S\in \bo$, one can define the \textit{trace} of $S$ by 
\begin{equation} \label{eq:trace} 
  \tr(S)=\sum_{k\in \N} \inner{Se_k}{e_k},
\end{equation}
where $\{e_k\}_{k\in \N}$ is an orthonormal basis for $L^2(\Rd)$. The Banach space $\tco$ of \textit{trace class operators} consists of those compact operators $S$ where $\tr(|S|)<\infty$, with norm $\|S\|_{\tco}=\tr(|S|)$. The trace in \eqref{eq:trace} defines a linear functional on $\tco$ that satisfies $\tr(ST)=\tr(TS)$, and the definition in \eqref{eq:trace} is independent of the orthonormal basis used \cite{Busch:2016}. By a celebrated theorem due to Lidskii, for $S\in \tco$, 
\begin{equation}\label{eq:lidskii}
	\tr(S)=\sum_{k=1}^{\infty}\lambda_k
\end{equation}
where the eigenvalues $\{\lambda_k\}_{k\in \N}$ of $S$ are counted with multiplicity \cite{Simon:2010wc}. 
\subsubsection{The Weyl transform} 
An important concept for associating operators on $L^2(\Rd)$ with functions on $\Rdd$ is the \textit{Weyl transform}. If $\phi\in \mathcal{S}'(\Rdd)$, then we define the Weyl transform $\phi^w$ as an operator $\mathcal{S}(\Rdd)\to \mathcal{S}'(\Rdd)$ by
\begin{equation*}
\inner{\phi^w \xi}{\psi}=\inner{\phi}{W(\psi,\xi)} \hskip 1em \text{ for } \xi,\psi \in \mathcal{S}(\Rd),
\end{equation*}
where the bracket denotes the action of $\mathcal{S}'(\Rd)$ as functionals on $\mathcal{S}(\Rd)$. We call $\phi$ the \textit{Weyl symbol} of the operator $\phi^w$. For more information on the Weyl transform in the same spirit as this short introduction, such as conditions to ensure $\phi^w\in \bo$, we refer to \cite{Grochenig:2001}.
\subsection{Quantum harmonic analysis} \label{sec:quantumharmonic}
This section introduces the theory of convolutions of operators and functions due to Werner \cite{Werner:1984}. For $z\in \R^{2d}$ and $A\in \bo$, we define the operator $\alpha_z(A)$ by
\begin{equation*}
  \alpha_z(A)=\pi(z)A\pi(z)^*.
\end{equation*}
It is easily confirmed that $\alpha_z\alpha_{z'}=\alpha_{z+z'}$, and we will informally think of $\alpha$ as a shift or translation of operators. 

Similarly we define the analogue of the involution $\check{f}(z):=f(-z)$ of a function, for an operator $A\in \bo$ by
\begin{equation*}
  \check{A}=PAP,
\end{equation*}
where $P$ is the parity operator $P\psi(x)=\psi(-x)$ for $\psi \in L^2(\R^d)$. 

Using $\alpha$, Werner defined a convolution operation between functions and operators \cite{Werner:1984}. If $f \in L^1(\R^{2d})$ and $S \in \tco$ we define the \textit{operator} $f\star S$ by
\begin{equation*}
  f\star S := S\star f := \int_{\R^{2d}}f(z)\alpha_z(S) \ dz
\end{equation*}
where the integral is interpreted in the weak sense by requiring that
\begin{equation*}
	\inner{(f\star S)\psi}{\xi})=\int_{\Rdd} f(z) \inner{\alpha_z(S)\psi}{\xi} \ dz, \hskip 1em \text{ for $\psi,\xi \in L^2(\Rd)$.}
\end{equation*}
Then $f\star S \in \tco$ and $\|f\star S \|_{\tco}\leq \|f \|_{L^1}\|S\|_{\tco}$ \cite[Prop. 2.5]{Luef:2017vs}.

For two operators $S,T \in \tco$, Werner defined the \textit{function} $S\star T$ by 
\begin{equation} \label{eq:defconv}
  S \star T(z) = \tr(S\alpha_z (\check{T})) \hskip 1em \text{for } z\in \R^{2d}.
\end{equation}

\begin{rem}
	The notation $\star$ may therefore denote either the convolution of two functions or the convolution of an operator with a function. The correct interpretation will be clear from the context.
\end{rem}
The following result relates the convolutions of operators to the standard convolutions of Weyl symbols. The statements follow by combining propositions 3.12 and 3.16(5) in \cite{Luef:2018}.
\begin{prop} \label{prop:weylconvolutions}
	Let $f\in L^1(\Rdd)$ and $S,T\in \tco$. Let $\phi_S$ and $\phi_T$ be the Weyl symbols of $S$ and $T$. Then
	\begin{enumerate}
		\item $S\star T(z)=\phi_S \ast \phi_T(z)$ for $z\in \Rdd$.
		\item The Weyl symbol of $f\star S$ is $f\ast \phi_S$.
	\end{enumerate}
	Here $\ast$ denotes the usual convolution of functions.
\end{prop}

The following result shows that $S\star T\in L^1(\R^{2d})$ for $S,T\in \tco$ and provides an important formula for its integral \cite[Lem. 3.1]{Werner:1984}. In the simplest case where $S$ and $T$ are rank-one operators, this formula is the so-called Moyal identity for the STFT\cite[p. 57]{Folland:1989}.
\begin{lem} 
\label{lem:werner}
Let $S,T \in \tco$. The function $z \mapsto S\star T(z)$ for $z\in \R^{2d}$ is integrable and $\|S\star T \|_{L^1} \leq \|S\|_{\tco} \|T\|_{\tco}$. Furthermore,
\begin{equation*}
	\int_{\R^{2d}} S\star T(z) \ dz = \tr(S)\tr(T).
\end{equation*}
\end{lem}

The convolutions can be defined on other $L^p$-spaces and Schatten $p$-classes by duality \cite{Luef:2017vs,Werner:1984}. As a special case we mention that \eqref{eq:defconv} defines a continuous function even when $T\in \bo$ \cite{Luef:2017vs}; in particular it is clear from \eqref{eq:defconv} that 
\begin{equation} \label{eq:identity}
	S\star I(z)=\tr(S)
\end{equation}
 for any $z\in \Rdd$ when $I$ is the identity operator and $S\in \tco$.
The convolutions of operators and functions are associative, a fact that is non-trivial since the convolutions between operators and functions can produce both operators and functions as output \cite{Luef:2017vs, Werner:1984}. Commutativity and bilinearity, however, follows straight from the definitions. We will also need the following simple property.
\begin{lem} \label{lem:inversion}
	Let $S\in \tco$ be a positive operator. If $\{\xi_n\}_{n\in \N}$ is an orthonormal basis for $L^2(\Rd)$, then
	 \begin{equation*}
		\sum_{n=1}^\infty S\star (\xi_n \otimes \xi_n)(z)=\tr(S), \hskip 1em \text{for any } z\in \Rdd.
	\end{equation*}
\end{lem}
\begin{proof}
 A simple calculation (see the proof of \cite[Thm. 1.5]{Luef:2017vs}) shows that 
\begin{equation*}
	S\star (\xi_n \otimes \xi_n)(z)=\inner{\check{S}\pi(-z)\xi_n}{\pi(-z)\xi_n}.
\end{equation*}
By proposition \ref{prop:singval}, $\check{S}$ has a spectral representation  $$\check{S}=\sum_{k=1}^\infty \lambda_k \varphi_k\otimes \varphi_k,$$ where  $\{\varphi_k\}_{k\in \N}$ is an orthonormal basis for $L^2(\Rd)$ and $\{\lambda_k\}_{k\in \N}$ are the eigenvalues of $\check{S}$. We insert this into the previous formula and apply Parseval's theorem to get
\begin{align*}
	\sum_{n=1}^\infty S\star (\xi_n \otimes \xi_n)(z)&=\sum_{n=1}^\infty \inner{\sum_{k=1}^\infty \lambda_k \varphi_k\otimes \varphi_k\pi(-z)\xi_n}{\pi(-z)\xi_n} \\
	&= \sum_{n=1}^\infty \sum_{k=1}^\infty \lambda_k \inner{\pi(-z)\xi_n}{\varphi_k}\inner{\varphi_k}{\pi(-z)\xi_n} \\
	&= \sum_{k=1}^\infty \lambda_k \sum_{n=1}^\infty \inner{\varphi_k}{\pi(-z)\xi_n} \inner{\pi(-z)\xi_n}{\varphi_k} \\
		&= \sum_{k=1}^\infty \lambda_k \inner{\varphi_k}{\varphi_k}  = \sum_{k=1}^\infty \lambda_k =\tr(\check{S})= \tr(S).\\
\end{align*}
The final line uses \eqref{eq:lidskii}, and that $\tr(\check{S})=\tr(PSP)=\tr(P^2 S)=\tr(S).$
Note that we used that $\pi(-z)$ is unitary to get that $\{\pi(-z)\xi_n\}_{n\in \N}$ is an orthonormal basis.
	
\end{proof}

The convolutions preserve positivity \cite[Lem. 4.1]{Luef:2018:physics} .
\begin{lem}
 \label{lem:positiveandidentity}
 If $S,T \in \bo$ are positive operators and $f$ is a positive function, then $f\star S$ is a positive operator and $S\star T$ is a positive function.	  
\end{lem}

\section{Cohen's class and mixed-state localization operators} \label{sec:genloc} A quadratic time-frequency distribution $Q$ is said to be of \textit{Cohen's class} if $Q$ is given by 
\begin{equation*}
  Q(\psi)=Q_{\phi}(\psi):=W(\psi,\psi)\ast \phi
\end{equation*}
for some $\phi \in \mathcal{S}^{\prime}(\R^{2d})$ \cite{Cohen:1966,Grochenig:2001}. In \cite{Luef:2018} we emphasized another way of defining Cohen's class, namely that $Q$ belongs to Cohen's class if $Q$ is given by
\begin{equation} \label{eq:defofcohen}
  Q(\psi)=Q_S(\psi) = \check{S}\star (\psi\otimes \psi),
\end{equation}
where $S:\mathcal{S}(\Rdd)\to \mathcal{S}^\prime(\Rdd)$ is a continuous linear operator. It can be shown using proposition \ref{prop:weylconvolutions} that these two definitions are equivalent \cite{Luef:2018}, since 
\begin{equation} \label{eq:coheneq}
Q_\phi=Q_S	\hskip 1em \text{ when } \phi^w=\check{S}.
\end{equation}
We will be particularly interested in $Q_S$ when $S$ is a \textit{positive trace class operator with $\tr(S)=1$}. In quantum mechanics, such operators are often called \textit{density operators}, and we will adopt this terminology in this paper. There is a simple characterization of those Cohen's class distributions $Q_S$ where $S$ is a density operator \cite{Luef:2018}.
\begin{prop} \label{prop:densitycohen}
	Let $S\in \bo$. $S$ is a density operator if and only if for any $\psi\in L^2(\Rd)$ $Q_S(\psi)$ is a positive function and  $\int_{\Rdd} Q_S(\psi)(z) \ dz=\|\psi\|^2_{L^2}. $
\end{prop}

In light of \eqref{eq:coheneq}, the set of Cohen's class distributions $Q_S$ with $S$ a density operator equals $\{Q_\phi:\phi\in \states\}$, where
\begin{equation*}
	\states := \{\phi \in \mathcal{S}(\Rdd): \phi^w \text{ is a density operator }\}.
\end{equation*}
\begin{rem}
Due to \eqref{eq:coheneq}, the operator $\check{S}$ will appear many times. The reader should therefore note that $\check{S}$ and $S$ are unitarily equivalent by definition, and share relevant properties such as positivity and trace.
\end{rem}
In \cite{Luef:2018} we also introduced the notion of a \textit{mixed-state localization operator}, which is an operator of the form $\chi_{\Omega}\star S$ where $S$ is a density operator and $\chi_\Omega$ the characteristic function of a domain $\Omega \subset \Rdd$. By definition $\chi_{\Omega}\star S$ acts on $\psi \in L^2(\Rd)$ by 
\begin{equation*}
	(\chi_\Omega \star S) \psi = \int_{\Omega} \pi(z)S\pi(z)^* \psi \ dz.
\end{equation*}
The simplest examples of density operators are given by the rank-one operators $\varphi \otimes \varphi$ for $\varphi\in L^2(\Rd)$ with $\|\varphi\|_{L^2}=1$. In this case, the Cohen class distribution $Q_{\varphi \otimes \varphi}$ is the \textit{spectrogram}:
\begin{equation} \label{eq:spectrogram}
  Q_{\varphi \otimes \varphi}(\psi)=(\check{\varphi} \otimes \check{\varphi}) \star (\psi\otimes \psi)=|V_{\varphi}\psi|^2,
\end{equation}
and the mixed-state localization operators $\chi_{\Omega}\star (\varphi \otimes \varphi)$ are the usual localization operators introduced by Daubechies \cite{Daubechies:1988}, which act on $\psi \in L^2(\Rd)$ by
\begin{equation*}
  (\chi_{\Omega}\star (\varphi \otimes \varphi))(\psi)=\int_{\Omega} V_{\varphi}\psi(z) \pi(z) \varphi \ dz. 
\end{equation*}

\begin{rem}
	In quantum mechanics, a rank-one operator $\varphi\otimes \varphi$ describes a so-called pure state of a system \cite{deGosson:2011wq}. More general states are called mixed states, and are described by density operators -- hence the terminology of mixed-state localization operators.
\end{rem}

\subsection{Notation for mixed-state localization operators} \label{sec:notation}
In order to fix notation, we briefly consider the spectral representation of mixed-state localization operators. If $\Omega \subset \Rdd$ is compact and $S$ is a density operator, we know from lemma \ref{lem:positiveandidentity} and section \ref{sec:quantumharmonic} that $\chi_\Omega \star S$ is a positive trace class operator. For the rest of the paper we will denote the eigenvalues of $\chi_\Omega \star S$ by $\{\lambda_k^\Omega\}_{k\in \N}$ and the orthonormal basis formed by its eigenfunctions by $\{h_k^\Omega\}_{k\in \N}$, thus the spectral representation is
\begin{equation} \label{eq:singvalgenloc}
	\chi_\Omega \star S = \sum_{k=1}^\infty \lambda_k^\Omega h_k^\Omega \otimes h_k^\Omega.
\end{equation}
We always assume that the eigenvalues are in decreasing order, i.e. $\lambda_1^\Omega\geq \lambda_2^\Omega \geq \dots$. \\
The function $S\star \check{S}$, for some operator $S\in \tco$, will play an important role in our results. To emphasize this, we introduce the notation $$\tilde{S}(z):=S\star \check{S}(z).$$
If $S$ is a density operator, it follows from section \ref{sec:quantumharmonic} that $\tilde{S}$ is a positive, continuous function such that $\int_{\Rdd} \tilde{S}(z) \ dz=\tr(S)\tr(S)=1.$ In the special case where $S=\varphi\otimes \varphi$ for some $\varphi \in L^2(\Rd)$, we get by \eqref{eq:spectrogram} that $\tilde{S}(z)=|V_\varphi \varphi(z)|^2$.
\subsection{A consequence of associativity}
As we have mentioned, the associativity of the convolutions introduced in section \ref{sec:quantumharmonic} is non-trivial. It leads to the following relation between Cohen's class distributions and mixed-state localization operators, see \cite[Lem. 4.1]{Abreu:2016} for an alternative proof for spectrograms.

\begin{prop} \label{prop:associativity}
	Let $S$ be a density operator and let $\Omega \subset \Rdd$ be a compact set. Then 
	\begin{equation*}
		\chi_\Omega \ast \tilde{S}(z)=\sum_{k=1}^\infty \lambda_k^\Omega Q_S(h_k^\Omega)(z), \hskip 1em \text{for } z\in \Rdd.
	\end{equation*}
\end{prop}
\begin{proof}
By the associativity of convolutions, we have that $\chi_\Omega \ast \tilde{S}=\chi_\Omega \ast (S\star \check{S})=(\chi_\Omega \star S)\star \check{S}$ in $L^1(\Rdd)\cap L^\infty(\Rdd)$. Now insert the  spectral representation from \eqref{eq:singvalgenloc}:
\begin{align*}
	(\chi_\Omega \star S)\star \check{S}&=\left(\sum_{k=1}^{\infty}\lambda_k^{\Omega} h_k^{\Omega}\otimes h_k^{\Omega}\right)\star \check{S}\\
	&= \sum_{k=1}^{\infty}\lambda_k^{\Omega} (h_k^{\Omega}\otimes h_k^{\Omega})\star \check{S} \\
	&= \sum_{k=1}^\infty \lambda_k^\Omega Q_S(h_k^\Omega).
\end{align*}
When moving to the second line, we have used that the spectral representation converges in the operator norm and that convolutions with a fixed operator is norm-continuous from $\bo$ to $L^\infty(\Rdd)$ \cite[Prop. 4.2]{Luef:2017vs}. Furthermore, $\chi_\Omega \ast (S\star \check{S})=(\chi_\Omega \star S)\star \check{S}$ holds pointwise since both sides are continuous functions -- the left side is the convolution of a bounded function $\chi_\Omega$ with $S\star \check{S}\in L^1(\Rdd)$, and the right side is the convolution of two trace class operators which is continuous by \cite[Prop. 3.3]{Luef:2017vs}.
\end{proof}
\subsection{Approximate identities for $L^1(\Rdd)$}
In the section we will obtain an approximate identity for $L^1(\Rdd)$ for each normalized trace class operator $S$. The following standard result is easily proved by straightforward calculations. 
\begin{prop}
	Let $\phi \in L^1(\Rd)$ satisfy $\int_{\Rd} \phi(x) \ dx=1.$ The family $\{\phi_R\}_{R>0}$ of normalized dilations of $\phi$ defined by $\phi_R(x)=R^{d}\phi(Rx)$ is an approximate identity for $L^1(\Rd).$
\end{prop}
As a consequence we obtain the following result, which is lemma 3 in \cite{Ramanathan:1994spec} when $\phi$ is positive. 
\begin{lem}
Let $\phi\in L^1(\Rd)$ be a function with $\int_{\Rdd} \phi(z) \ dz=1$, and let $\Omega\subset \Rd$ be a compact domain. Then 
\begin{equation*}
  \frac{1}{R^{d}}\int_{R\Omega}\int_{R\Omega} \phi(x-x^{\prime}) \ dx \ dx^{\prime}\to |\Omega|
\end{equation*}
as $R\to \infty$.
\end{lem}
\begin{proof}
	By the previous proposition we know that the family $\{\phi_R\}_{R>0}$ is an approximate identity for $L^1(\Rd).$ In particular we have that $\chi_{\Omega}\ast \phi \to \chi_{\Omega}$ in $L^1(\Rd)$ as $R\to \infty$. Since $\psi \mapsto \int_{\Rd}\psi(x) \chi_{\Omega}(x) \ dx$ is a linear functional on $L^1(\Rd)$, we get as a consequence that $\int_{\Omega} \chi_{\Omega}\ast \phi_R(x) \ dx \to |\Omega|$ as $R\to \infty$. It only remains to show that $\int_{\Omega} \chi_{\Omega}\ast \phi_R(x) \ dx$ equals the left hand side in the statement of the theorem:
	\begin{align*}
  \int_{\Omega} \chi_{\Omega}\ast \phi_R \ dx &= \int_{\Omega}\int_{\Rd} \chi_{\Omega}(x^\prime)R^d\phi(R(x-x^\prime)) \ dx^\prime \ dx \\
  &= R^d\int_{\Omega}\int_{\Omega} \phi(R(x-x^\prime)) \ dx^\prime \ dx \\
  &=\frac{1}{R^d}\int_{R \Omega}\int_{R \Omega} \phi(u-v) \ du \ dv,
\end{align*}
where we have introduced the new variables $u=Rx$ and $v=Rx^\prime$.
\end{proof}

This allows us to introduce an important class of approximate identities based on trace class operators.
\begin{cor} \label{cor:appid}
	Let $S\in \tco$ be an operator with $\tr(S)=1$. The functions $\{\tilde{S}_R\}_{R>0}$ form an approximate identity for $L^1(\Rdd)$ and 
	\begin{equation*}
  \frac{1}{R^{2d}}\int_{R\Omega}\int_{R\Omega} \tilde{S}(z-z^{\prime}) \ dz \ dz^{\prime}\to |\Omega|
\end{equation*}
as $R\to \infty$ for any compact domain $\Omega \subset \Rdd$.
\end{cor}
\begin{proof}
	By lemma \ref{lem:werner}, $\tilde{S}=S\star \check{S}\in L^1(\Rdd)$ and $\int_{\Rdd} S\star \check{S}(z) \ dz=\tr(S)\tr(\check{S})=1$, hence the result follows from the previous lemma and proposition.
\end{proof}

\section{The eigenvalues of mixed-state localization operators} \label{sec:spectraltheory}
In this section we will be interested in the eigenvalues of mixed-state localization operators $\chi_{R\Omega}\star S$ as $R\to \infty$, where $R\Omega =\{R z:z\in \Omega\}$. In the case of localization operators, corresponding to $S=\varphi \otimes \varphi$ for $\varphi \in L^2(\Rd)$, the following behaviour of the eigenvalues $\{\lambda_k^{R\Omega}\}_{k\in \N}$ of $\chi_{R\Omega}\star (\varphi \otimes \varphi)$ has been established in \cite{Ramanathan:1994spec,Feichtinger:2001}:
\begin{equation} \label{eq:eigenvalues}
  \frac{\#\{k:\lambda_k^{R\Omega}>1-\delta\}}{R^{2d}|\Omega|}\to 1 \text{ as $R\to \infty$}.
\end{equation}
To show that this holds for the eigenvalues $\{\lambda_k^{R\Omega}\}_{k\in \N}$ of any mixed-state localization operator $\chi_{R\Omega}\star S$, we need a few lemmas. 

\begin{lem}
If $S$ is a density operator and $\Omega \subset \Rdd$ a compact domain, the eigenvalues of $\chi_\Omega\star S$ satisfy $0\leq \lambda_k^\Omega \leq 1.$	
\end{lem}
\begin{proof}
	As we saw in section \ref{sec:notation}, $\chi_\Omega \star S$ is a positive operator, so its eigenvalues are non-negative. By equation (9) in \cite{Luef:2018}, $\inner{\chi_\Omega \star S \psi}{\psi}=\int_{\Omega} Q_S(\psi)(z)\ dz$ for $\psi \in L^2(\Rd).$ If we let $\psi$ be the eigenvector $h_k^\Omega$, proposition \ref{prop:densitycohen} now gives $$\lambda_k^\Omega = \int_{\Omega} Q_S(h_k^\Omega)(z)\ dz\leq \int_{\Rdd} Q_S(h_k^\Omega)(z)\ dz=\|h_k^\Omega\|_{L^2}^2=1.$$
	\end{proof}

\begin{lem} \label{lem:traces}
	Let $\Omega\subset \R^{2d}$ be a compact domain, and let $S\in \tco$. 
	\begin{align*}
  \tr(\chi_{\Omega}\star S)&= \sum_{k=1}^{\infty} \lambda_k^\Omega =|\Omega|\tr(S), \\
  \tr((\chi_{\Omega}\star S)^2)&=\int_{\Omega}\int_{\Omega} \tilde{S}(z-z^{\prime}) \ dz \ dz^{\prime}.
\end{align*}
\end{lem}
\begin{proof}
	The formula $\tr(\chi_{\Omega}\star S)= \sum_{k=1}^{\infty} \lambda_k^\Omega$ is Lidskii's theorem from \eqref{eq:lidskii}.  To prove $\tr(\chi_{\Omega}\star S) =|\Omega|\tr(S)$, we note that \eqref{eq:identity} says that $(\chi_\Omega \star S) \star I(z)=\tr(\chi_\Omega \star S)$ for any $z\in \Rdd.$ However, by associativity of convolutions and $S\star I(z)=\tr(S)$ we also have that 
	\begin{align*}
		(\chi_\Omega \star S) \star I(z)&= \chi_\Omega \ast (S \star I)(z) \\
		&= \int_{\Rdd} \chi_{\Omega} (z') (S \star I)(z-z') \ dz' = \tr(S) |\Omega|.
	\end{align*}
	
	 For the second part, note that $T\star \check{T}(0)=\tr(T^2)$ for any $T\in \tco$ by the definition of convolution of operators. In particular\footnote{The alert reader will note that we use $ (\chi_\Omega \star S) \parcheck=\check{\chi_\Omega}\star \check{S}$. See \cite{Skrettingland:2017} for the simple proof.} $(\chi_\Omega \star S)\star (\check{\chi_\Omega} \star \check{S})(0)=\tr((\chi_\Omega \star S)^2)$.  Hence, using associativity and commutativity of convolutions,
	 \begin{align*}
	 	\tr((\chi_\Omega \star S)^2)&=\chi_\Omega \ast (\check{\chi_\Omega}\ast (S\star  \check{S}))(0) \\
	 	&= \int_\Omega\check{\chi_\Omega}\ast   (S\star \check{S})(-z') \ dz' \\
	 	&= \int_\Omega \int_{\Rdd} \chi_\Omega(-z) (S\star \check{S})(-z'-z) \ dz \ dz' \\
	 	&= \int_\Omega \int_{\Omega} (S\star \check{S})(z-z') \ dz \ dz',
	 \end{align*}
	 where we substituted $z\mapsto -z$ in the last line.
	 \end{proof}
\begin{rem}
	For rank-one operators $S=\varphi\otimes \varphi$ for $\varphi \in L^2(\Rd)$ these formulas are well known and used to obtain the profile of the eigenvalues of localization operators, see for instance \cite{Ramanathan:1994spec,Feichtinger:2001,Abreu:2016}. The approach used to obtain the second formula in these papers uses the reproducing kernel Hilbert space associated with the short-time Fourier transform. Our approach does not rely on this property of the STFT, which allows us to prove the result for general trace class operators. 
\end{rem}
The following is a generalization of \cite[Lem 3.3]{Abreu:2016} to mixed-state localization operators. Our proof follows the proof from that paper, which is based on the approach in \cite{Feichtinger:2001}.
\begin{lem} \label{lem:eigenvalues}
	Let $S$ be a density operator, let $\Omega \subset \Rdd$ be a compact domain and fix $\delta \in (0,1)$. Then
	\begin{equation*}
  \left|\# \{ k\geq 1:\lambda_k^{\Omega}>1-\delta\}-|\Omega|  \right|\leq \max \left\{\frac{1}{\delta},\frac{1}{1-\delta}\right\} \left|\int_{\Omega} \int_{\Omega} \tilde{S}(z-z') dz dz^{\prime}-|\Omega|\right|
\end{equation*}
\end{lem}
\begin{proof}
	 Following \cite{Abreu:2016} we define the function 
	\begin{equation*}
  G(t):= 
     \begin{cases}
       -t &\quad\text{if } 0\leq t\leq 1-\delta\\
       1-t &\quad\text{if } 1-\delta < t \leq 1.\\ 
     \end{cases}
\end{equation*}
  We may apply $G$ to the eigenvalues in the spectral representation \eqref{eq:singvalgenloc}  to obtain a new operator $G(\chi_\Omega \star S)$:
\begin{equation*}
  G(\chi_{\Omega}\star S)=\sum_{k=1}^{\infty} G(\lambda_k^\Omega) h_k^\Omega \otimes h_k^\Omega. 
\end{equation*}
Since $\chi_\Omega \star S$ is trace class, $\{\lambda_k^\Omega\}_{k=1}^\infty\in \ell^1$. As $\sum_{k=1}^\infty \lambda_k^\Omega=|\Omega|$, only finitely many $\lambda_k^\Omega$ can satisfy $\lambda_k^\Omega >1-\delta$, and it follows that $\{G(\lambda_k^\Omega)\}_{k=1}^\infty\in \ell^1$ because $|G(t)|= |t|$ for $t\in [0,1-\delta]$. Hence $G(\chi_{\Omega}\star S)$ is a trace class operator with trace
\begin{align*}
  \tr(G(\chi_{\Omega}\star S))&=\sum_{k=1}^{\infty} G(\lambda_k^\Omega) \\
  &= \# \{k:\lambda_k^\Omega > 1-\delta\}-\sum_{k=1}^{\infty} \lambda_k^\Omega \\
  &= \# \{ k\geq 1:\lambda_k^{\Omega}>1-\delta\}-|\Omega|.
\end{align*}
Therefore
\begin{align*}
  \left|\# \{ k\geq 1:\lambda_k^{\Omega}>1-\delta\}-|\Omega|  \right|&= \left| \tr(G(\chi_{\Omega}\star S)) \right| \\
  &\leq \tr( |G|(\chi_{\Omega}\star S))\\ 
  &\leq \max \left\{\frac{1}{\delta},\frac{1}{1-\delta}\right\}\tr \left(\chi_{\Omega}\star S - \left( \chi_{\Omega}\star S\right)^2\right),
  \end{align*}
  where we have used $|G(t)|\leq \max \{\frac{1}{\delta},\frac{1}{1-\delta}\}(t-t^2)$ for $t\in [0,1]$. The final result follows from inserting the expressions for $\tr(\chi_\Omega\star S)$ and $\tr((\chi_\Omega \star S)^2)$ from lemma \ref{lem:traces}.
\end{proof}

The following is the main result of this section, which shows that \eqref{eq:eigenvalues} is valid for mixed-state localization operators.
\begin{thm} \label{thm:locopeigenvalues}
	Let $S$ be a density operator, let $\Omega\subset \Rdd$ be a compact domain and fix $\delta\in (0,1)$. If $\{\lambda_k^{R\Omega}\}_{k\in \N}$ are the eigenvalues of $\chi_{R\Omega}\star S$, then
	\begin{equation*}
  \frac{\#\{k:\lambda_k^{R\Omega}>1-\delta\}}{R^{2d}|\Omega|}\to 1 \text{ as $R\to \infty$}.
\end{equation*}
\end{thm}
\begin{proof}
	By the previous lemma,
	{\footnotesize
\begin{equation*}
		\left|\# \{ k\geq 1:\lambda_k^{R\Omega}>1-\delta\}-R^{2d}|\Omega|  \right|\leq \max \left\{\frac{1}{\delta},\frac{1}{1-\delta}\right\} \left|\int_{R\Omega} \int_{R\Omega} \tilde{S}(z-z') dz dz^{\prime}-R^{2d}|\Omega|\right|.
\end{equation*}}
Hence if we divide by $R^{2d}|\Omega|$ {\footnotesize
\begin{equation*}
  \left|\frac{\# \{ k\geq 1:\lambda_k^{R\Omega}>1-\delta\}}{R^{2d}|\Omega|}-1  \right|\leq \max \left\{\frac{1}{\delta},\frac{1}{1-\delta}\right\} \frac{1}{|\Omega|} \left|\frac{1}{R^{2d}}\int_{R\Omega} \int_{R\Omega} \tilde{S}(z-z') dz dz^{\prime}-|\Omega|\right|.
\end{equation*}}
The result now follows from corollary \ref{cor:appid}.
\end{proof}

\section{Accumulated Cohen class distributions}
For any density operator $S$ and domain $\Omega \subset \Rdd$, we define an associated \textit{accumulated Cohen class distribution} by  \begin{equation*}
  \rho^S_{\Omega}(z):=\sum_{k=1}^{A_\Omega} Q_S(h_k^\Omega) \hskip 1em \text{ for } z\in \Rdd,
\end{equation*}
where $A_{\Omega}=\lceil |\Omega| \rceil$ and $h_k^\Omega$ are the eigenfunctions of $\chi_{\Omega}\star S$. Note that $\rho_\Omega^S$ may also be written as a convolution of operators, since 
\begin{equation*}
  \rho^S_{\Omega}=\sum_{k=1}^{A_\Omega} \check{S}\star (h_k^\Omega \otimes h_k^\Omega)=\check{S}\star \sum_{k=1}^{A_\Omega}  (h_k^\Omega \otimes h_k^\Omega).
\end{equation*}
As a consequence, lemma \ref{lem:inversion} gives that 
	$\rho_\Omega^S(z)\leq 1$  for any  $z\in \Rdd$,  since $\{h_k^\Omega\}_{k\in \N}$ is an orthonormal basis and $$\rho_\Omega^S(z)=\sum_{n=1}^{A_\Omega} \check{S}\star (h_k^\Omega\otimes h_k^\Omega)(z)\leq \sum_{n=1}^\infty \check{S}\star (h_k^\Omega \otimes h_k^\Omega)(z)=\tr(S)=1.$$

In \cite{Abreu:2016}, Abreu et al. prove results showing that when $Q_S$ is a spectrogram, $\rho_\Omega^S$ is an approximation of the characteristic function $\chi_\Omega$. We will show that their results hold when $S$ is \textit{any} density operator.  Our presentation and proofs follow those in \cite{Abreu:2016}. The proofs will typically consist of two parts: the easy part is to show that the function $\chi_\Omega \ast \tilde{S}$ approximates $\chi_\Omega$. The more intricate part is to show that $\chi_\Omega \ast \tilde{S}$ also approximates $\rho_\Omega^S$.  We start by generalizing \cite[Lem. 4.2, 4.3]{Abreu:2016}.
\begin{lem} \label{lem:abreu}
	Let $\Omega \subset \Rdd$ be a compact domain and define
	\begin{equation*}
  E(\Omega)=1-\frac{\sum_{k=1}^{A_{\Omega} }\lambda_k^\Omega}{|\Omega|}.
\end{equation*}
Then
\begin{equation*}
	\frac{1}{|\Omega|} \|\rho^S_{\Omega}-\chi_{\Omega}\ast \tilde{S}\|_{L^1}\leq \left( \frac{1}{|\Omega|}+2E(\Omega) \right),
\end{equation*}
and $$E(R \Omega)\to 0 \text{ as } R\to \infty.$$
\end{lem}
\begin{proof}
	 Using lemma \ref{lem:werner} and the associativity of convolutions, we find that
	\begin{align*}
  \|\rho^S_{\Omega}-\chi_{\Omega}\ast (S\star \check{S})\|_{L^1}&=\left\|\left(\sum_{k=1}^{A_\Omega}  h_k^\Omega \otimes h_k^\Omega \right)\star \check{S}-(\chi_{\Omega}\star S)\star \check{S}\right\|_{L^1}  \\
  &\leq \left\|\sum_{k=1}^{A_\Omega}  h_k^\Omega \otimes h_k^\Omega-\chi_{\Omega}\star S\right\|_{\tco}\left\|\check{S}\right\|_{\tco} \\
  &=\left\|\sum_{k=1}^{A_{\Omega}}h_k^{\Omega}\otimes h_k^{\Omega}-\sum_{k=1}^{\infty}\lambda_k^{\Omega} h_k^{\Omega}\otimes h_k^{\Omega}\right\|_{\tco} \\
  &= \sum_{k=1}^{A_{\Omega}} (1-\lambda_k^{\Omega})+\sum_{k=A_{\Omega}+1}^{\infty}\lambda_k^{\Omega}.
\end{align*}
We have expanded $\chi_{\Omega}\star S$ using the spectral representation \eqref{eq:singvalgenloc}, and the last equality uses that $\|T\|_\tco$ is the sum of the eigenvalues for positive operators $T\in \tco$. Since $\sum_{k=1}^{\infty}\lambda_k^{\Omega}=|\Omega|$, we further get that
\begin{align*}
  \sum_{k=1}^{A_{\Omega}} (1-\lambda_k^{\Omega})+\sum_{k=A_{\Omega}+1}^{\infty}\lambda_k^{\Omega}&=|\Omega|+A_{\Omega}-2\sum_{k=1}^{A_{\Omega}}\lambda_k^{\Omega} \\
  &= \left(A_{\Omega}-|\Omega|\right)+2\left(|\Omega|-\sum_{k=1}^{A_{\Omega}}\lambda_k^{\Omega}\right)\\
  &\leq 1+2E(\Omega)|\Omega|.
\end{align*}
To prove that $E(R \Omega)\to 0$ as $R\to \infty$, we will pick $\delta \in (0,1)$ and find an upper bound of $E(R \Omega)$ in terms of $\frac{\# \{k: \lambda_k^{R\Omega}>1-\delta\}}{|\Omega|}$ -- an application of theorem \ref{thm:locopeigenvalues} will then give the desired result. For a fixed $\delta \in (0,1)$ and domain $\Omega$, we define $$l_{\delta}(\Omega)=\min\{A_{\Omega},\# \{k:\lambda_k^{\Omega}>1-\delta\}\}.$$
By definition $l_{\delta}(\Omega)\leq A_{\Omega}$,  and since the eigenvalues $\lambda_k^\Omega$ are arranged in decreasing order we see that $\lambda_k^\Omega > 1-\delta$ for $k\leq l_\delta(\Omega).$ Using this we estimate that
\begin{align*}
  E(\Omega)&=1-\frac{\sum_{k=1}^{A_{\Omega}}\lambda_{k}^\Omega}{|\Omega|} \\
  &\leq 1-\frac{\sum_{k=1}^{l_{\delta}(\Omega)}\lambda_{k}^\Omega}{|\Omega|} \\
  &\leq 1-(1-\delta)\frac{l_\delta(\Omega)}{|\Omega|},
\end{align*}
where we have also used that the eigenvalues $\lambda_{k}^\Omega$ are non-negative. Note that we always have $E(\Omega)\geq 0$, since $\sum_{k=1}^\infty \lambda_k^\Omega=|\Omega|$.
If we replace the domain $\Omega$ by the new domain $R \Omega$ in the previous estimate and insert the definition of $l_{\delta}(R \Omega)$, we obtain 
\begin{equation*}
  0\leq E(R \Omega)\leq 1-(1-\delta)\min\left\{\frac{A_{R\Omega}}{R^{2d}|\Omega|}, \frac{\# \{k:\lambda_k^{R \Omega}>1-\delta\}}{R^{2d}|\Omega|}\right\}.
\end{equation*}
By definition of $A_{\Omega}$ we know that $\frac{A_{R \Omega}}{|\Omega|R^{2d}}\geq 1$, hence we get the estimate
\begin{equation} \label{eq:Ebound}
  0\leq E(R \Omega)\leq 1-(1-\delta)\min\left\{1, \frac{\# \{k:\lambda_k^{R \Omega}>1-\delta\}}{R^{2d}|\Omega|}\right\}.
\end{equation}

The behaviour of the term $\frac{\# \{k:\lambda_k^{R \Omega}>1-\delta\}}{R^{2d}|\Omega|}$ is described by theorem \ref{thm:locopeigenvalues}, which says that this fraction approaches 1 as $R\to \infty$. Therefore $$0\leq \limsup_{R\to \infty} E(R \Omega)\leq 1-(1-\delta)=\delta,$$ and by picking $\delta$ arbitrarily close to $0$ we see that in fact $E(R \Omega)\to 0$ as $R\to \infty$.

\end{proof}

\subsection{Asymptotic convergence of accumulated Cohen class distributions}
We are now ready to prove the generalization of \cite[Thm. 4.3]{Abreu:2016} -- the asymptotic convergence of accumulated Cohen's class distributions to the characteristic function of the domain.
\asymptotic*
\begin{proof}
We will use the estimate 
	\begin{align*}
  \|\rho^S_{R \Omega}(R\hskip 0.1em\boldsymbol{\cdot})-\chi_{\Omega}\|_{L^1}\leq \|\rho^S_{R \Omega}(R\hskip 0.1em\boldsymbol{\cdot})-\chi_{\Omega}\ast \tilde{S}_R\|_{L^1} + \|\chi_{\Omega}\ast \tilde{S}_R-\chi_{\Omega}\|_{L^1},
\end{align*}
where $\tilde{S}_R(z)=R^{2d}\tilde{S}(Rz)$. The second term converges to 0 as $R\to \infty$ by corollary \ref{cor:appid}. To bound the first term, we note that a straightforward calculation using a change of variable gives that $\chi_\Omega \ast\tilde{S}_R(z)=\chi_{\Omega}\ast (S\star \check{S})_R(z)=\chi_{R \Omega} \ast (S\star \check{S})(Rz)$. Hence we find, with $z^\prime=Rz$, that 
\begin{align*}
  \|\rho^S_{R \Omega}(R\hskip 0.1em\boldsymbol{\cdot})-\chi_{\Omega}\ast \tilde{S}_R\|_{L^1}&=\int_{\Rdd} |\rho^S_{R \Omega}(Rz)-\chi_{R \Omega} \ast (S\star \check{S})(Rz)| \ dz  \\
  &=\frac{1}{R^{2d}} \int_{\Rdd}|\rho^S_{R \Omega}(z^\prime)-\chi_{R \Omega} \ast (S\star \check{S})(z^\prime)| \ dz^\prime \\
  &\leq \frac{1}{R^{2d}}+2 E(R\Omega)|\Omega|,
\end{align*}
where the last inequality is lemma \ref{lem:abreu}. By the same lemma, this expression converges to 0 as $R\to \infty.$
\end{proof}

	The above result shows  that the domain $\Omega$ is uniquely determined by $\rho^S_{R \Omega}$ as $R\to \infty,$ i.e. from knowledge of $S$ and the first $A_{R\Omega}=\lceil |R\Omega| \rceil$ eigenfunctions of $\chi_{R\Omega}\star S$ for infinitely many $R$. In \cite{Luef:2018} we used a Tauberian theorem for operators due to Werner\cite{Werner:1984} to establish certain conditions on $S$, formulated in terms of a Fourier transform for operators, that guarantee that $\Omega$ can be recovered from only $\chi_\Omega \star S$. The next two sections will show that we may \textit{estimate} $\Omega$ from $\chi_\Omega \star S$, but make no claim that $\Omega$ is determined by $\chi_\Omega \star S$ for any density operator $S$.

\subsection{Non-asymptotic approximation by accumulated Cohen class distributions}
The bounds for the non-asymptotic convergence of accumulated Cohen class distributions will depend on the size of the perimeter of the domain $\Omega \subset \Rdd$. To quantify the size of the perimeter, we will use the variation of its characteristic function $\chi_\Omega$. Hence we define $$|\partial \Omega|=Var(\chi_\Omega)$$ for a domain $\Omega \subset \Rdd.$ We say that $\Omega$ has finite perimeter if $\chi_\Omega$ has bounded variation. The only way this will enter our considerations is via the following lemma, which is proved in \cite[Lem. 3.2]{Abreu:2016} where the reader may also find some more relevant discussion and references regarding functions of bounded variation.
\begin{lem}  \label{lem:boundedvariation}
	Let $f\in L^1(\Rd)$ have bounded variation, and let $\varphi \in L^1(\Rd)$ satisfy $\int_{\Rd}\varphi(z) \ dz = 1.$ Then $$\|f\ast \varphi-f\|_1\leq Var(f)\int_{\Rd}|x| |\varphi(x)| \ dx,$$
	where $|x|$ denotes the Euclidean norm on $\Rd$.
\end{lem}
We also define a subset $M_{op}^*$ of density operators by 
\begin{equation*} 
  M_{op}^*=\{S\in \tco : S\geq 0, \tr(S)=1  \text{ and } \int_{\Rdd} \tilde{S}(z) |z| \ dz < \infty  \},
\end{equation*}
where $|z|$ is the Euclidean norm of $z$, with the associated norm $$\|S\|^2_{M_{op}^*}=\int_{\Rdd} \tilde{S}(z) |z| \ dz.$$
This norm lets us bound the approximation of $\chi_\Omega$ by $\chi_\Omega \ast \tilde{S}$, since lemma \ref{lem:boundedvariation} gives
\begin{equation} \label{eq:l1boundeasy}
	\|\chi_\Omega -\chi_\Omega \ast \tilde{S}\|_{L^1} \leq |\partial \Omega| \|S\|_{M^*_{op}}^2.
\end{equation}
 When $Q_S$ is a spectrogram, i.e. $S=\varphi \otimes \varphi$ for some $\varphi \in L^2(\Rdd)$ by \eqref{eq:spectrogram}, the norm $\|S\|_{M_{op}^*}^2$ becomes $\int_{\Rdd}|V_\varphi \varphi(z)|^2 |z| \ dz$, which is the norm $\|\varphi\|_{M^*}$ introduced in \cite{Abreu:2016} for accumulated spectrograms.
We now prove the generalization of \cite[Prop. 3.4]{Abreu:2016}.
\begin{lem} \label{lem:opnormbound}
Let $\Omega \subset \Rdd$ be a compact domain with finite perimeter and $S\in M^*_{op}(\Rd)$. If $\delta \in (0,1)$, then
\begin{equation*}
\left| \#\{k:\lambda_k^{\Omega}>1-\delta\}-|\Omega| \right|\leq \max \left\{\frac{1}{\delta},\frac{1}{1-\delta}\right\}\|S\|_{M^*}^2|\partial \Omega|
\end{equation*}
\end{lem}
\begin{proof}
	By lemma \ref{lem:eigenvalues}, it suffices to bound the expression
	\begin{equation*}
  \left|\int_{\Omega} \int_{\Omega} \tilde{S}(z-z') dz dz^{\prime}-|\Omega|\right|.
\end{equation*}
We may rewrite this expression as
\begin{align*}
 \left|\int_{\Omega} \int_{\Rdd}\chi_{\Omega}(z) \tilde{S}(z-z') dz dz^{\prime}-|\Omega|\right|&=\left|\int_{\Omega} \chi_{\Omega}\ast \tilde{S}(z^{\prime})  dz^{\prime}-\int_{\Omega}\chi_{\Omega}(z^{\prime})\ dz^{\prime}\right| \\
 &= \left|\int_{\Omega} \chi_{\Omega}\ast \tilde{S}(z^{\prime})  -\chi_{\Omega}(z^{\prime})\ dz^{\prime}\right| \\
 &\leq\int_{\Rdd} \left| \chi_{\Omega}\ast \tilde{S}(z^{\prime})  -\chi_{\Omega}(z^{\prime})\right|  \ dz^{\prime}\\
 &= \|\chi_{\Omega}\ast \tilde{S}-\chi_{\Omega}\|_{L^1},
\end{align*}
where we have used $\tilde{S}(z-z')=\tilde{S}(z'-z)$ to write the left summand as a convolution with $\chi_{\Omega}$. This relation holds since $S\star \check{S}(-z) = \check{S} \star \check{\check{S}}(z)= \check{S} \star S(z)=S\star \check{S}(z)$, see \cite[Lem. 4.7]{Skrettingland:2017}. The result now follows from lemma \ref{lem:eigenvalues} and \eqref{eq:l1boundeasy}.
\end{proof}

The following $L^1$-bound generalizes \cite[Thm. 1.4]{Abreu:2016} to general $S\in M^*_{op}$.
\begin{thm}\label{thm:lonebound}
	If $S\in M_{op}^*$ and $\Omega \subset \Rdd$ is a compact domain with finite perimeter, then $$\frac{1}{|\Omega|}\|\rho_\Omega^S-\chi_\Omega \ast \tilde{S}\|_{L^1}\leq \left(\frac{1}{|\Omega|}+4\|S\|_{M^*_{op}} \sqrt{\frac{|\partial \Omega|}{|\Omega|}} \right).$$
\end{thm}
\begin{proof}
	From lemma \ref{lem:abreu},
 $$	\frac{1}{|\Omega|} \|\rho^S_{\Omega}-\chi_{\Omega}\ast \tilde{S}\|_{L^1}\leq \left( \frac{1}{|\Omega|}+2E(\Omega) \right).$$
 We will prove the theorem by proving the estimate $E(\Omega)\leq  2 \|S\|_{M_{op}^*} \sqrt{\frac{|\partial \Omega|}{|\Omega|}}$, which generalizes \cite[Lem. 4.3]{Abreu:2016}. We therefore jump back to our estimate in \eqref{eq:Ebound}, which was the estimate for $E(R\Omega)$ we obtained when we did not assume $S\in M_{op}^*$. For $R=1$ this equation gives
 \begin{equation} \label{eq:proofofasymptotic}
  0\leq E(\Omega)\leq 1-(1-\delta) \frac{\# \{k:\lambda_k^{ \Omega}>1-\delta\}}{|\Omega|}.
\end{equation}
To bound this expression, we note that lemma \ref{lem:opnormbound} gives 
\begin{equation*}
	\frac{\# \{k:\lambda_k^{ \Omega}>1-\delta\}}{|\Omega|} \geq1- \max \left\{\frac{1}{\delta},\frac{1}{1-\delta}\right\}\|S\|_{M^*}^2\frac{|\partial \Omega|}{|\Omega|}.
\end{equation*}
Inserting this estimate into \eqref{eq:proofofasymptotic} and setting $\delta = \|S\|_{M_{op}^*}\sqrt{\frac{|\partial \Omega|}{|\Omega|}}$ now gives the desired estimate -- we refer to the proof of \cite[Lem 4.3]{Abreu:2016} for the details.
\end{proof}
As a corollary, one can derive an estimate for $\|\rho_\Omega^S-\chi_\Omega \|_{L^1}$. We return to this question in section \ref{sec:sharp}.
\subsection{Weak $L^2$-convergence of accumulated Cohen class distributions}
Finally, we show that the weak-$L^2$ bounds for $\rho_\Omega^S-\chi_\Omega$ in \cite[Thm 1.5]{Abreu:2016} hold in the more general case where $S$ is a density operator. Following the proof in \cite{Abreu:2016} we start by proving a technical lemma. 

\begin{lem}
	If $S\in M_{op}^*$ and $\Omega \subset \Rdd$ is a compact domain with finite perimeter such that $\|S\|^2_{M_{op}^*}|\partial \Omega|\geq 1$, then for any $\delta >0$ $$\left| \left\{ z\in \Rdd : \left| \rho_\Omega^S(z)-\chi_\Omega\ast \tilde{S}(z)  \right|> \delta \right\} \right|\lesssim \frac{1}{\delta^2}\|S\|^2_{M_{op}^*}|\partial \Omega|.$$
\end{lem}

\begin{proof}
	By proposition \ref{prop:associativity} we find that
	\begin{align*}
  |\rho^S_{\Omega}(z)-\chi_{\Omega}\ast \tilde{S}(z)|&= \left| \sum_{k=1}^{A_{\Omega}} Q_S(h_k^\Omega)(z)-\sum_{k=1}^{\infty}\lambda_k^{\Omega} Q_S(h_k^{\Omega})(z) \right|  \\
  &\leq \sum_{k=1}^{\infty}\mu_k Q_S(h_k^\Omega)(z),
\end{align*}
where we have introduced $\mu_k=\lambda_k^\Omega$ for $k> A_\Omega$ and $\mu_k=1-\lambda_k^\Omega$ for $k\leq A_\Omega$. To obtain our desired bound, we will split this sum into three parts. Following the lead of the proof in \cite[Prop. 4.4]{Abreu:2016}, we assume that $0<\delta \leq \frac{1}{2}$ and define 
\begin{align*}
  a_\delta &:= \#\{k:\lambda_k^\Omega> 1-\delta\}, \\ 
    b_\delta &:= \#\{k:\lambda_k^\Omega>\delta\}. \\
\end{align*}
Then let 
\begin{align*}
  a_\delta^\prime &:= \min \{a_\delta, A_\Omega\}, \\ 
    b_\delta^\prime &:= \max\{b_\delta, A_\Omega\}. \\
\end{align*}
Now note that $\sum_{k=1}^\infty Q_S(h_k)(z)=1$ for all $z\in \Rdd$ by lemma \ref{lem:inversion} and that $\mu_k\leq \delta$ for $k\notin [a_\delta^\prime+1,b_\delta^\prime]$, hence
\begin{align*}
	\sum_{k=1}^{\infty}\mu_k Q_S(h_k^\Omega)(z)&=\sum_{k=1}^{a_\delta^\prime} \mu_k Q_S(h_k^\Omega)(z) + \sum_{a_\delta^\prime+1}^{b_\delta^\prime} \mu_k Q_S(h_k^\Omega)(z)+ \sum_{k=b_\delta^\prime+1}^{\infty} \mu_k Q_S(h_k^\Omega)(z) \\
	&\leq 2\delta + \sum_{a_\delta^\prime+1}^{b_\delta^\prime} \mu_k Q_S(h_k^\Omega)(z).
\end{align*}
As a consequence we clearly get $$\left| \left\{ z\in \Rdd : \left| \rho_\Omega^S(z)-\chi_\Omega\ast \tilde{S}(z)  \right|> 3\delta \right\} \right|\leq \left| \left\{ z\in \Rdd : \sum_{a_\delta^\prime+1}^{b_\delta^\prime} \mu_k Q_S(h_k^\Omega)(z) > \delta \right\} \right|.$$ To control this expression, one may use lemma \ref{lem:opnormbound} and the assumptions $0<\delta \leq \frac{1}{2}$, $\|S\|^2_{M_{op}^*}|\partial \Omega|\geq 1$ to get (see \cite[Prop. 4.4]{Abreu:2016} for details) $$0\leq b_\delta^\prime-a_\delta^\prime\lesssim \frac{1}{\delta}\|S\|_{M_{op}^*} |\partial \Omega|.$$
By using $0\leq \mu_k \leq 1$ and $\|Q_S(h_k^{\Omega})\|_{L^1}=\|h_k^\Omega\|_{L^2}^2=1$ (see \eqref{eq:defofcohen} and lemma \ref{lem:werner}), we then get
 \begin{align*}
	\left| \left\{ z\in \Rdd : \sum_{a_\delta^\prime+1}^{b_\delta^\prime} \mu_k Q_S(h_k^\Omega)(z) > \delta \right\} \right|&\leq \frac{1}{\delta} \left\|\sum_{a_\delta^\prime+1}^{b_\delta^\prime} \mu_k Q_S(h_k^\Omega)\right\|_{L^1} = \frac{1}{\delta} \sum_{a_\delta^\prime+1}^{b_\delta^\prime} \mu_k \left\| Q_S(h_k^\Omega)\right\|_{L^1} \\
	 &\leq \frac{1}{\delta } \sum_{a_\delta^\prime+1}^{b_\delta^\prime} 1 \lesssim \frac{1}{\delta^2}\|S\|_{M_{op}^*} |\partial \Omega|.
\end{align*}

 The substitution $\delta \mapsto \frac{\delta}{3}$ proves the result for $0<\delta \leq \frac{3}{2}$, and the result is trivial for $\delta >1$ since we always have the bound $ |\rho^S_{\Omega}(z)-\chi_{\Omega}\ast \tilde{S}(z)|\leq \sum_{k=1}^{\infty}\mu_k Q_S(h_k^\Omega)(z)\leq \sum_{k=1}^{\infty} Q_S(h_k^\Omega)(z)=1$ by lemma \ref{lem:inversion}.
\end{proof}

\nonasymptotic*
\begin{proof}
	By the previous lemma we have the weak-$L^2$ bound $$\left| \left\{ z\in \Rdd : \left| \rho_\Omega^S(z)-\chi_\Omega\ast \tilde{S}(z)  \right|> \delta/2 \right\} \right|\lesssim \frac{1}{\delta^2}\|S\|^2_{M_{op}^*}|\partial \Omega|,$$ and we obviously have the weak-$L^1$ bound 
	$$\left| \left\{z\in \Rdd : |\chi_\Omega \ast \tilde{S}(z)-\chi_\Omega(z)|> \delta/2 \right\} \right|\leq \frac{2}{\delta}\|\chi_\Omega \ast \tilde{S}(z)-\chi_\Omega(z)\|_{L^1}\leq \frac{1}{\delta} \|S\|^2_{M_{op}^*}|\partial \Omega|,$$
	where the last bound is lemma \ref{lem:boundedvariation}. Combining these bounds, we get
	$$\left| \left\{ z\in \Rdd : \left| \rho_\Omega^S(z)-\chi_\Omega(z)  \right|> \delta \right\} \right|\lesssim \frac{1}{\delta^2}\|S\|^2_{M_{op}^*}|\partial \Omega|+ \frac{1}{\delta} \|S\|^2_{M_{op}^*}|\partial \Omega|.$$ When $\delta\leq 2$ we have that $1/\delta \leq 2/\delta^2,$ so the result is proved in this case. In fact, this is the only case we need to consider, as $\rho_\Omega^S(z), \chi_\Omega (z)\leq 1$ implies $\{z\in \Rdd : |\rho_\Omega^S(z)-\chi_\Omega(z)|>\delta\}=\emptyset$ for $\delta >2$.
\end{proof}
\section{Sharp bounds for accumulated Cohen's class distributions} \label{sec:sharp}
As a simple consequence of theorem \ref{thm:lonebound} one can derive the bound
\begin{equation*}
	\|\chi_\Omega-\rho_\Omega^S\|_{L^1}\lesssim \sqrt{|\partial \Omega||\Omega|}
\end{equation*}
for $S,\Omega$ satisfying the assumptions of that theorem and $|\Omega|\geq 1$, see \cite[Cor. 5.1]{Abreu:2016} for a proof when $Q_S$ is a spectrogram. In \cite{Abreu:2017}, Abreu et al. were able to improve this bound in the case of spectrograms to 
\begin{equation} \label{eq:sharpbound}
	\|\chi_\Omega-\rho_\Omega^S\|_{L^1}\lesssim |\partial \Omega|.
\end{equation}
The very elegant proof of \eqref{eq:sharpbound} in \cite{Abreu:2017} exploits the spectral theory of localization operators. Since section \ref{sec:spectraltheory} indicates that the spectral theory is largely the same for \textit{generalized} localization operators, we will be able to prove \eqref{eq:sharpbound} for general density operators $S$ based on the same arguments.
\begin{thm} \label{thm:sharpbound}
	Fix $\epsilon >0$. If $S\in M_{op}^*$ and $\Omega \subset \Rdd$ is a compact domain with finite perimeter satisfying $|\partial \Omega|\geq \epsilon$, then
	\begin{equation*}
		\|  \rho_\Omega^S -\chi_\Omega\|_{L^1}\leq (1/\epsilon+2\|S\|_{M_{op}^*}^2) |\partial \Omega|.
	\end{equation*}
\end{thm}

\begin{proof}
	To estimate the left hand side, we will split the integral into two parts. First note that since $0\leq \rho_\Omega^S(z) \leq 1$ for any $z\in \Rdd$,
	\begin{align*}
		\int_{\Omega} |\rho_\Omega^S(z) - \chi_\Omega(z)| \ dz &= \int_\Omega (1-\rho_\Omega^S(z)) \ dz \\
		&= |\Omega|-\int_\Omega \sum_{k=1}^{A_\Omega} \check{S}\star (h_k^\Omega \otimes h_k^\Omega)(z)  \ dz \\
		&= |\Omega|-\sum_{k=1}^{A_\Omega} \int_{\Rdd} \chi_\Omega(z) \left( \check{S}\star  (h_k^\Omega \otimes h_k^\Omega)\right)(z)  \ dz \\
		&= |\Omega|-\sum_{k=1}^{A_\Omega} \lambda_k^\Omega.
	\end{align*}
	The final equality uses a relation between convolutions and duality, namely the fact that $\inner{\chi_\Omega}{\check{S}\star (h_k^\Omega\otimes h_k^\Omega)}_{L^\infty,L^1}=\inner{\chi_{\Omega}\star S}{h_k^\Omega \otimes h_k^\Omega}_{\bo,\tco}$, where the bracket denotes duality. See \cite{Skrettingland:2017} for a verification. By using the eigenfunctions $\{h_k^\Omega\}_{k=1}^\infty$ as the orthonormal basis to calculate the trace, one easily finds that $\inner{\chi_{\Omega}\star S}{h_k^\Omega \otimes h_k^\Omega}_{\bo,\tco}=\tr((\chi_{\Omega}\star S) h_k^\Omega \otimes h_k^\Omega)=\lambda_k^\Omega$.  The other part of the integral satisfies
	\begin{align*}
		\int_{\Rdd \setminus \Omega} |\rho_\Omega^S(z) - \chi_\Omega(z)| \ dz &= \int_{\Rdd \setminus \Omega} \rho_\Omega^S(z) \ dz \\
		&= \int_{\Rdd} \rho_\Omega^S(z) \ dz - \int_\Omega \rho_\Omega^S(z) \ dz \\
		&=\sum_{k=1}^{A_\Omega} \int_{\Rdd} \check{S}\star (h_k^\Omega \otimes h_k^\Omega)(z)  \ dz - \sum_{k=1}^{A_\Omega} \lambda_k^\Omega \\
		&= A_\Omega - \sum_{k=1}^{A_\Omega} \lambda_k^\Omega\leq 1+|\Omega|-\sum_{k=1}^{A_\Omega} \lambda_k^\Omega, 
	\end{align*}
	where we have used lemma \ref{lem:werner} to calculate $\int_{\Rdd} \check{S}\star (h_k^\Omega \otimes h_k^\Omega)(z)  \ dz=1$, and used our expression for  $\int_\Omega \rho_\Omega^S(z) \ dz$ from the previous calculation. In total
	\begin{equation} \label{eq:sharp}
		\int_{\Rdd} |\rho_\Omega^S(z) - \chi_\Omega(z)| \ dz\leq 1+2\left( |\Omega|-\sum_{k=1}^{A_\Omega} \lambda_k^\Omega \right).
	\end{equation}
	To bound $|\Omega|-\sum_{k=1}^{A_\Omega} \lambda_k^\Omega$ we will look at $\tr(\chi_\Omega \star S)-\tr((\chi_\Omega \star S)^2).$ On the one hand it follows easily from lemma \ref{lem:traces} that 
	\begin{align*}
		\tr(\chi_\Omega \star S)-\tr((\chi_\Omega \star S)^2) &= \int_{\Omega} \left( 1- \chi_{\Omega}\ast \tilde{S}(z) \right) \ dz \\
		&\leq \|\chi_\Omega \ast \tilde{S} - \chi_\Omega \|_{L^1} \\
		&\leq |\partial \Omega | \|S\|_{M_{op}^*}^2, 
	\end{align*}
	where the last inequality is lemma \ref{lem:boundedvariation}. On the other hand we know that $\tr(\chi_\Omega \star S)=\sum_{k=1}^\infty \lambda_k^\Omega = |\Omega|$ and $\tr((\chi_\Omega \star S)^2)=\sum_{k=1}^{\infty} \left(\lambda_k^\Omega\right)^2$, which leads to the following estimate:
	\begin{align*}
		\tr(\chi_\Omega \star S)-\tr((\chi_\Omega \star S)^2)&= \sum_{k=1}^{A_\Omega} \lambda_k^\Omega (1-\lambda_k^\Omega)+\sum_{k=A_\Omega+1}^\infty  \lambda_k^\Omega (1-\lambda_k^\Omega) \\
		&\geq \lambda_{A_\Omega}^\Omega  \sum_{k=1}^{A_\Omega} (1-\lambda_k^\Omega) +(1-\lambda_{A_\Omega}^\Omega)\sum_{k=A_\Omega+1}^\infty  \lambda_k^\Omega \\
		&= \lambda_{A_\Omega}^\Omega A_\Omega - \lambda_{A_\Omega}^\Omega |\Omega|+\sum_{k=A_\Omega +1}^{\infty} \lambda_k^\Omega \\
		&= \lambda_{A_\Omega}^\Omega(A_\Omega-|\Omega|) + |\Omega|-\sum_{k=1}^{A_\Omega}\lambda_k^\Omega \\ 
		&\geq |\Omega|-\sum_{k=1}^{A_\Omega}\lambda_k^\Omega.
	\end{align*}
	We therefore have $|\Omega|-\sum_{k=1}^{A_\Omega}\lambda_k^\Omega \leq \tr(\chi_\Omega \star S)-\tr((\chi_\Omega \star S)^2)\leq |\partial \Omega | \|S\|_{M_{op}^*}^2,$ and inserting this into \eqref{eq:sharp} gives us $$\int_{\Rdd} |\rho_\Omega^S(z) - \chi_\Omega(z)| \ dz\leq \left(1/\epsilon+2 \|S\|_{M_{op}^*}^2 \right)|\partial \Omega |$$ when we also use the assumption $|\partial \Omega |/\epsilon\geq 1$. 
\end{proof}

\subsection{Sharpness of the bound}

By considering Euclidean balls $B(z,R)=\{z'\in \Rdd : |z|<R\}$, it was shown in \cite{Abreu:2017} that \eqref{eq:sharpbound} gives a sharp bound for the convergence of accumulated spectrograms. As we will now show (see theorem \ref{thm:sharprate}), the same is true when the spectrogram is replaced with the Cohen class distribution $Q_S$ for  $S\in M_{op}^* .$ Our approach is inspired by \cite{DeMari:2002}, which deals with the case of spectrograms using the associated reproducing kernel Hilbert spaces. These Hilbert spaces are not available for general density operators $S$, so our proofs must instead rely on techniques from quantum harmonic analysis. In the terminology of \cite{Feichtinger:2014,Feichtinger:2015}, the following result gives an expression for the \textit{projection functional} applied to $\chi_\Omega \star S$.
\begin{lem} \label{lem:tracedifference}
	Let $S$ be a density operator and $\Omega \subset \Rdd$ a compact domain. Then $$\tr(\chi_\Omega \star S)-\tr((\chi_\Omega \star S)^2)=\int_{\Omega}\int_{\Rdd\setminus \Omega } \tilde{S}(z-z')\ dz' dz. $$
\end{lem}
\begin{proof}
From lemma \ref{lem:traces} we have that 
\begin{align*}
	\tr(\chi_\Omega \star  S)&=\int_{\Rdd} \chi_\Omega (z)\ dz \\
	 \tr((\chi_{\Omega}\star S)^2)&=\int_{\Omega}\int_{\Omega} \tilde{S}(z-z^{\prime}) \ dz' \ dz.
\end{align*}
 In order to combine these two formulas, we note that $$\int_{\Rdd} \tilde{S}(z-z')\ dz= \int_{\Rdd} S\star \check{S}(z-z') \ dz'=\tr(S)\tr(S)=1$$ for each $z\in \Rdd$ by lemma \ref{lem:werner}. Hence we can in fact write 
\begin{align*}
	\tr(\chi_\Omega \star S)&=\int_{\Rdd} \chi_\Omega(z) \int_{\Rdd} \tilde{S}(z-z') \ dz' \ dz \\
	&= \int_{\Omega} \int_{\Rdd} \tilde{S}(z-z') \ dz' \ dz.
\end{align*}

We may now combine our formulas to get that 
\begin{align*}
	\tr(\chi_\Omega \star S)-\tr((\chi_\Omega \star S)^2)&= \int_{\Omega} \int_{\Rdd} \tilde{S}(z-z') \ dz' \ dz - \int_{\Omega}\int_{\Omega} \tilde{S}(z-z^{\prime}) \ dz' \ dz \\
	&= \int_\Omega \int_{\Rdd\setminus \Omega} \tilde{S}(z-z') \ dz' \ dz.
\end{align*}

\end{proof}
We will also need the following technical consequence of the continuity of $\tilde{S}$.
\begin{lem} \label{lem:lowerbound}
	Let $S$ be a density operator. There exist constants $r_S>0$ and $m>0$ such that whenever $r\leq r_S$ 
	\begin{equation} \label{eq:lowerbound}
		\tilde{S}(z-z') \geq \frac{m}{|B(0,r) |} \int_{\Rdd} \chi_{B(z'',r)}(z)\chi_{B(z'',r)}(z') \ dz''
	\end{equation}
	for all $z,z'\in \Rdd$.
\end{lem}

\begin{proof}
	The function $\tilde{S}=S\star \check{S}$ is continuous, positive and satisfies $S\star \check{S}(0)=\tr(S^2)>0$. Let $m=\tr(S^2)/2>0$. By continuity of $\tilde{S}$ at the origin, there must exist a constant $\delta>0$ such that $S\star \check{S}(z)>m$ whenever $z\in B(0,\delta).$ Now let $r_S=\delta/2$, and consider the integral $$\int_{\Rdd} \chi_{B(z'',r)}(z)\chi_{B(z'',r)}(z') \ dz''$$ for $r\leq r_S.$ We note that the integrand is zero whenever $z-z'\notin B(0,2r)$. When $z-z'\in B(0,2r) \subset B(0,\delta)$ we know by construction of $\delta$ that $S\star \check{S}(z-z')\geq m$. We may also estimate that for any $z,z'\in \Rdd$    
	\begin{align*}
		\int_{\Rdd} \chi_{B(z'',r)}(z)\chi_{B(z'',r)}(z') \ dz''&\leq \int_{\Rdd} \chi_{B(z'',r)}(z) \ dz'' \\
		&= |B(0,r)|.
	\end{align*}  
	Hence \eqref{eq:lowerbound} holds: if $z-z'\notin B(0,2r)$ it holds trivially as the integrand is zero, and if $z-z'\in B(0,2r)$, we know that $S\star \check{S}(z-z')\geq m$ and the integral is bounded from above by $B(0,r)$.     
\end{proof}
	The previous two results lead to a lower bound for the projection functional for mixed-state localization operators with $\Omega=B(0,R).$
	\begin{prop} \label{prop:lowerbound}
	Let $S$ be a density operator. Then there exists a constant $C_S$ such that 
	\begin{equation*}
		\tr(\chi_{B(0,R)}\star S)-\tr((\chi_{B(0,R)} \star S)^2)\geq C_S R^{2d-1}, \hskip 1em \text{ for }R>1.
	\end{equation*}
\end{prop}
\begin{proof}
	Let $r=\min \{r_S,1\}$. By lemma \ref{lem:tracedifference} we know that 
	$$\tr(\chi_{B(0,R)} \star S)-\tr((\chi_{B(0,R)} \star S)^2)=\int_{{B(0,R)}}\int_{\Rdd\setminus {B(0,R)} } \tilde{S}(z-z')\ dz' dz, $$
	and by inserting the estimate from lemma \ref{lem:lowerbound}  we get 
	{\footnotesize
	\begin{equation*}
		\tr(\chi_{B(0,R)} \star S)-\tr((\chi_{B(0,R)} \star S)^2)\geq \frac{m}{|B(0,r)|} \int_{{B(0,R)}}\int_{\Rdd\setminus {B(0,R)} }  \int_{\Rdd} \chi_{B(z'',r)}(z)\chi_{B(z'',r)}(z') \ dz'' dz' dz.
	\end{equation*} 
	}
	By changing the order of integration the right hand side becomes
	\begin{equation} \label{eq:proofsharp}
		\frac{m}{|B(0,r)|}  \int_{\Rdd} |B(0,R)\cap B(z'',r)|\ |\left(\Rdd\setminus B(0,R)\right)\cap B(z'',r)| \ dz''.
	\end{equation}
	Now assume that $z''$ lies in the strip in $\Rdd$ defined by $R-r/2\leq |z''|\leq R+r/2$. A simple estimate shows that both $B(0,R)\cap B(z'',r)$ and $\left(\Rdd\setminus B(0,R)\right)\cap B(z'',r)$ must contain a ball of radius $r/4$ in this case, so that $$|B(0,R)\cap B(z'',r)|\ |\left(\Rdd\setminus B(0,R)\right)\cap B(z'',r)|\geq |B(0,r/4)|^2 .$$

The expression in \eqref{eq:proofsharp} is therefore bounded from below by 
	{\footnotesize
	\begin{align*}
		|B(0,r/4)|^2\frac{m}{|B(0,r)|}  \int_{R-r/2 \leq |z''|\leq R+r/2}  \ dz'' &= |B(0,r/4)|^2\frac{m}{|B(0,r)|} C_d \left((R+r/2)^{2d} -(R-r/2)^{2d} \right) \\
		&\geq |B(0,r/4)|^2\frac{m}{|B(0,r)|} C_d 2d r R^{2d-1},
	\end{align*}
	}
	which finishes the proof by setting $C_S:=|B(0,r/4)|^2\frac{m}{ |B(0,r)|} C_d 2 d r$. Here $C_d$ is the measure of the unit sphere in $\Rdd$, and the fact that $(R+r/2)^{2d} -(R-r/2)^{2d} \geq 2 d r R^{2d-1}$ is a simple consequence of the binomial theorem.
\end{proof}
Using these results we may now prove the desired sharpness of \eqref{eq:sharpbound} with exactly the same arguments that were used to prove it for accumulated spectrograms in \cite{Abreu:2017}.
\sharpness*

\begin{proof}
	Since $|\partial B(0,R)|=C_d R^{2d-1}$, where $C_d$ is the measure of the unit sphere in $\Rdd$, the upper bound follows from theorem \ref{thm:sharpbound} with $\epsilon=1/C_d$. For the lower bound we will bound $\| \rho_{B(0,R)}^S-\chi_{B(0,R)} \|_{L^1}$ by $\tr(\chi_\Omega\star S)-\tr((\chi_\Omega\star S)^2$ from below, which will imply the result by proposition \ref{prop:lowerbound}. In the proof of theorem \ref{thm:sharpbound} we derived the equalities 
	\begin{align*}
		\int_{\Omega} |\rho^S_\Omega(z) - \chi_\Omega(z)| \ dz &= |\Omega|-\sum_{k=1}^{A_\Omega} \lambda_k^\Omega \\
		\int_{\Rdd \setminus \Omega} |\rho^S_\Omega(z) - \chi_\Omega(z)| \ dz &= A_\Omega - \sum_{k=1}^{A_\Omega} \lambda_k^\Omega,
	\end{align*}
which together give us -- when using $\sum_{k=1}^{\infty} \lambda_k^\Omega =|\Omega|$ by lemma \ref{lem:traces} -- that
\begin{align*}
	\| \rho_{B(0,R)}^S-\chi_{B(0,R)} \|_{L^1}&= |\Omega|-\sum_{k=1}^{A_\Omega} \lambda_k^\Omega + A_\Omega - \sum_{k=1}^{A_\Omega} \lambda_k^\Omega \\
	&= \sum_{k=1}^{\infty} \lambda_k^\Omega - \sum_{k=1}^{A_\Omega} \lambda_k^\Omega +\sum_{k=1}^{A_\Omega} (1-\lambda_k^\Omega) \\
	&= \sum_{k=A_\Omega +1}^\infty \lambda_k^\Omega  +\sum_{k=1}^{A_\Omega} (1-\lambda_k^\Omega) \\
	&\geq \sum_{k=A_\Omega +1}^\infty \lambda_k^\Omega(1-\lambda_k^\Omega)  +\sum_{k=1}^{A_\Omega}\lambda_k^\Omega  (1-\lambda_k^\Omega) \\
	&= \sum_{k=1}^{\infty} \lambda_k^\Omega-\sum_{k=1}^\infty \left(\lambda_k^\Omega\right)^2 = \tr(\chi_\Omega\star S)-\tr((\chi_\Omega\star S)^2.
\end{align*}
As mentioned, the result now follows from proposition \ref{prop:lowerbound}.
\end{proof}
\begin{rem}
	In \cite{Abreu:2017}, the previous result is stated for spectrograms when $R>0.$ We have only obtained the result for $R>1$, but this is not because we consider a more general setting. In fact, the proof for the upper bound in \cite{Abreu:2017} is simply theorem \ref{thm:sharpbound} with $\epsilon=1$, which needs the assumption $|\partial \Omega|\geq 1$. This is clearly not satisfied for arbitrarily small $R$.
\end{rem}
\section{Examples and other perspectives}
We now turn to examples of Cohen's class distributions such the theory of accumulated Cohen's class distributions works, namely those given by $$Q_S(\psi)=\check{S}\star(\psi \otimes \psi)$$ for some density operator $S$. As we have mentioned, the definition above is equivalent to the more standard definition of Cohen's class, where $\phi \in \mathcal{S}'(\Rdd)$ defines a Cohen's class distribution by $$Q_\phi (\psi)=\phi \ast W(\psi,\psi).$$ In fact, we introduced the set $\states$ such that $\phi\in \states$ if and only if  $Q_\phi=Q_S$ for some density operator $S$.   We will therefore look for functions $\phi$ that belong to $\states$.

\subsubsection{A Weyl symbol characterization of $M_{op}^*$}
Before we look at the examples, we reformulate the definition of $M_{op}^*$. By proposition \ref{prop:weylconvolutions}, the condition for $S\in M_{op}^*$ is, in terms of the Weyl symbol $\phi$ of $\check{S}$, that $\phi \in \states$ and

\begin{equation} \label{eq:equivalentcondition}
  \int_{\Rdd} \phi \ast \check{\phi}(z) |z| \ dz <\infty. 
\end{equation}

\subsection{Examples}
\subsubsection{Spectrograms}
If $S=\varphi \otimes \varphi$ for $\varphi\in L^2(\Rd)$, such that $\check{S}$ has Weyl symbol $\phi=W(\check{\varphi},\check{\varphi})$, then $S$ is a density operator and by \eqref{eq:spectrogram} $Q_S$ is the spectrogram $Q_S(\psi)=|V_\varphi \psi|^2$. A calculation using the definition of convolutions of operators reveals that $\tilde{S}=|V_\varphi \varphi|^2$, so that $S\in M_{op}^*$ if and only if $\int_{\Rdd} |V_\varphi \varphi |^2(z) |z|<\infty.$ This is the setting considered in the theory of accumulated spectrograms \cite{Abreu:2016,Abreu:2017}.

\subsubsection{Schwartz functions} \label{sec:examplesschwartz}
If $\phi$ belongs to the Schwartz space $\mathcal{S}(\Rdd)$, then it is well-known \cite[Prop. 286]{deGosson:2011wq} that $\phi^w$ is a trace class operator with $$\tr(\phi^w)=\int_{\Rdd} \phi(z) \ dz.$$
Hence any suitably normalized $\phi\in \mathcal{S}(\Rdd)$ gives us an operator $\phi^w$ that is trace class with $\tr(\phi^w)=1,$ and it is also clear from \eqref{eq:equivalentcondition} that $\phi^w\in M_{op}^*$ in this case. The problem of determining whether $\phi^w$ is positive is much more difficult. The classical conditions on $\phi$ for $\phi^w$ to be a positive operator are the KLM-conditions \cite{Kastler:1965,Loupias:1966,Loupias:1967}, see also the more recent results in \cite{Cordero:2017}. In the case where $\phi$ is a normalized, generalized Gaussian
 \begin{equation*}
  \phi(z)=2^d \frac{1}{\det(M)^{1/4}}e^{-z^T\cdot  M \cdot z}    \text{ for $z\in \Rdd$},
\end{equation*}
for some $2d\times 2d$-matrix $M$, it is known \cite{Gracia:1988,Mukunda:1987} that the Weyl transform $\phi^w$ is a positive operator if and only if 
\begin{equation*}
  M=S^T \Lambda S,
\end{equation*}
where $S$ is a symplectic matrix and $\Lambda$ is diagonal matrix of the form \[\Lambda=\text{diag}(\lambda_1,\lambda_2,\dots,
\lambda_d,\lambda_1,\lambda_2,\dots,\lambda_d)\] with $0< \lambda_i\leq 1$. Hence $\phi^w$ is a density operator in this case, and the theory of accumulated Cohen's class distributions will work for all such Gaussians. One should note that $Q_\phi$ is \textit{not} a spectrogram for many of these Gaussians  \cite{deGosson:2017,Gracia:1988}. Versions of these results have been obtained several times, see section 4.2 in \cite{Janssen:1997} and references therein, and they are also linked with the symplectic structure of the phase space \cite{deGosson:2009}.
\subsubsection{A nonexample: the Wigner distribution} 
The prototype of a Cohen's class distribution is the Wigner distribution $W(\psi,\psi)$.
By a result due to Grossmann \cite{Grossmann:1976}, the Wigner distribution corresponds to $S=2^d P$ in \eqref{eq:defofcohen}, i.e. $$W(\psi,\psi)=Q_P(\psi)= 2^d P \star (\psi \otimes \psi),$$
see \cite{Luef:2018} for a proof.  
The parity operator $P$ is not a density operator, and so our approach does not apply to the Wigner distribution. In fact, it has been shown that the operators $\chi_\Omega \star P$ are never trace class for a non-trivial domain $\Omega$ \cite[Prop. 11]{Ramanathan:1993}. As a consequence, the methods exploited in this paper, which often consider the sum of eigenvalues of such operators, will fail for the Wigner distribution.

\subsection{Generating new examples from old}
Checking whether a given function $\phi$ belongs to  $\states$ is in general a non-trivial task. However, using quantum harmonic analysis we can use our examples of $\phi\in \states$ to obtain new elements of $\states$.
\begin{lem}
	Let $S$ be a density operator, and let $f\in L^1(\Rdd)$ be a positive function such that $\int_{\Rdd} f(z) \ dz =1.$ Then $f\star S$ is a density operator.
\end{lem}

\begin{proof}
	By lemma \ref{lem:positiveandidentity}, $f\star S$ is a positive operator, and the same proof as for the first part of lemma \ref{lem:traces} gives that $$\tr\left(f\star S \right)=\int_{\Rdd} f(z) \ dz\ \tr(S)=1.$$

\end{proof}
Using associativity of convolutions we see that the Cohen's class distribution associated to $\check{f} \star S  $ is given by
\begin{equation*}
	Q_{\check{f}\star S}(\psi)=(\check{f} \star S)\parcheck \star (\psi \otimes \psi)=f \ast Q_S(\psi).
\end{equation*}

\begin{cor}
	Let $\phi \in \states$, and let $f\in L^1(\Rdd)$ be a positive function such that $\int_{\Rdd} f(z) \ dz =1.$ Then $f\ast \phi \in \states$, and the Cohen's class distributions associated to $\phi$ and $f\ast \phi$ are related by
	\begin{equation} \label{convolutioncohen}
		Q_{f\ast \phi} (\psi)=f\ast Q_\phi(\psi)\hskip 1em \text{ for } \psi\in L^2(\Rd).
	\end{equation}
\end{cor}
\begin{proof}
	We know from proposition \ref{prop:weylconvolutions} that $(f\ast \phi)^w=f\star \phi^w$, and as $\phi^w$ is a density operator by assumption the previous lemma gives that $f \ast \phi\in \states$. By definition
	\begin{equation*}
		Q_{f\ast \phi}(\psi)=(f\ast \phi) \ast W(\psi,\psi)=f\ast (\phi \ast W(\psi,\psi))=f\ast Q_\phi(\psi).
	\end{equation*}
	 \end{proof}

In particular, this works when $Q_\phi$ is a spectrogram, i.e. $Q_\phi(\psi)=|V_\varphi \psi |^2$ for some $\varphi \in L^2(\Rdd)$. We then obtain the new Cohen's class distribution 
\begin{equation*} 
	Q_{\phi \ast f}(\psi)=f\ast |V_{\varphi}\psi|^2.
\end{equation*}
The non-asymptotic bounds on the convergence of accumulated Cohen's class distributions $\rho_{\Omega}^S$ in theorems \ref{thm:lonebound}, \ref{thm:weakl2} and \ref{thm:sharpbound} depend on the quantity $\|S\|_{M_{op}^*}$. We are therefore interested in how this quantity changes when $S$ is replaced by the new density operator $\check{f}\star S$ discussed above, or equivalently when $Q_\phi$ is replaced by $f\ast Q_{ \phi}$.

\begin{prop}
	Let $S\in M_{op}^*$, and let $f$ be a positive function such that $\int_{\Rdd}f(z) \ dz=1$ and $\int_{\Rdd}f(z)|z|\ dz <\infty$. Then $\check{f}\star S\in M_{op}^*$, and $$\|\check{f}\star S\|_{M^*_{op}}^2\leq \|S\|_{M_{op}^*}^2+2 \int_{\Rdd} f(z) |z| \ dz.$$
\end{prop}
\begin{proof}
	We begin by proving a general result. Assume that $g,h$ are positive functions on $\Rdd$ such that $\int_{\Rdd} g(z)= \int_{\Rdd} h(z)\ dz=1$. Then
	\begin{align*}
		\int_{\Rdd} g\ast h(z) |z| \ dz&= \int_{\Rdd}\int_{\Rdd} g(z')h(z-z') |z| \ dz' \ dz \\
		&= \int_{\Rdd} g(z') \int_{\Rdd} h(z'') |z''+z'| \ dz'' \ dz' \hskip 3em (z'':=z-z') \\
		&\leq \int_{\Rdd} g(z') \int_{\Rdd} h(z'') \left(|z''|+|z'|\right) \ dz'' \ dz' \\
		&= \int_{\Rdd} g(z') \ dz' \int_{\Rdd} h(z'') |z''| \ dz'' + \int_{\Rdd} g(z') |z'| \ dz' \int_{\Rdd} h(z'') \ dz'' \\
		&=  \int_{\Rdd} h(z) |z| \ dz+  \int_{\Rdd} g(z) |z| \ dz.
	\end{align*}
	Now note that $$\|\check{f}\star S\|_{M_{op}^*}^2=\int_{\Rdd} (\check{f}\star S) \star (\check{f}\star S)\parcheck (z) |z| \ dz = \int_{\Rdd} (\check{f}\ast f)  \ast (\check{S}\star S)(z) |z| \ dz,$$ where we have used $(\check{f}\star S )\parcheck=f\star \check{S}$ and the commutativity and associativity of convolutions. The functions $g=\check{f}\ast f$ and $h=S\star \check{S}$ satisfy the assumptions of the calculation above, since they are positive functions by lemma \ref{lem:positiveandidentity}, $\int_{\Rdd} S\star \check{S}(z) \ dz =\tr(S)\tr(\check{S})=1$ by lemma \ref{lem:werner} and $\int_{\Rdd} \check{f}\ast f (z) \ dz=\left(\int_{\Rdd} f(z) \ dz\right)^2 =1$ by a simple calculation using Tonelli's theorem. So we apply our calculation with these functions, and find that
	\begin{align*}
		\|\check{f}\star S\|_{M_{op}^*}^2 &\leq \int_{\Rdd} \check{f}\ast f(z) |z| \ dz+  \int_{\Rdd} S\star \check{S}(z) |z| \ dz \\
		&= \|S\|_{M_{op}^*}^2+\int_{\Rdd} f\ast \check{f}(z) |z| \ dz.
	\end{align*}  
	Furthermore, if we pick $g=\check{f}$ and $h=f$, we get $\int_{\Rdd} \check{f}\ast f(z) |z| \ dz \leq 2 \int_{\Rdd} f(z) |z| \ dz$. If we insert this into the estimate above, our result follows.
\end{proof}

\begin{rem}
	The idea of smoothing a time-frequence distribution $Q$ by taking convolutions with a function $f$ on $\Rdd$, as in \eqref{convolutioncohen}, is useful in practice \cite{Janssen:1997}. In a sense, this is the idea behind Cohen's class, which by definition consists of smoothened versions of the Wigner distribution. Janssen mentions the case where $Q$ has ''rapidly alternating positive and negative values``\cite[p. 3]{Janssen:1997}, where smoothing can remove this behaviour. In fact, we saw in example \ref{sec:examplesschwartz} that convolving the Wigner distribution with a Gaussian $\phi$ produces a positive distribution $Q_\phi$.
	\end{rem}

Another simple way of obtaining new examples is to consider convex combinations. If $\phi_n \in \states$ for each $1\leq n \leq N$ and $\{\lambda_n\}_{n=1}^N$ is a sequence of nonnegative numbers with $\sum_{n=1}^N \lambda_n=1$, then $$\phi:= \sum_{n=1}^N \lambda_n \phi_n$$ also belongs to $\states$ since $\phi^w=\sum_{n=1}^N \lambda_n \phi_n^w$. Using the definition of positivity for operators it is trivial to check that $\phi^w$ is positive, and $$\tr(\phi^w)=\sum_{n=1}^N \lambda_n \tr(\phi_n^w)=\sum_{n=1}^N \lambda_n=1.$$

\subsection*{Acknowledgment}
This work was finished while F.L.\ wAS visiting NuHAG at the Faculty of Mathematics at the University of Vienna, who is  thankful for the hospitality.

\end{document}